\documentclass[11pt,reqno]{amsart}
 \usepackage[margin=1in]{geometry} 
\usepackage{amsmath,amsthm,amssymb,amsfonts,graphicx, fullpage, amsmath, url, verbatim, amssymb}

\usepackage[T1]{fontenc}
\usepackage{esint}
\usepackage{hyperref}
\usepackage{subcaption}
\usepackage{bbm}
\usepackage{color}
\usepackage{afterpage}

\newcommand{\R}{\mathbb{R}}

\allowdisplaybreaks

\numberwithin{equation}{section}
\newtheorem{theorem}{Theorem}[section]
\newtheorem{lemma}[theorem]{Lemma}

\newtheorem{proposition}[theorem]{Proposition}

\def\XXint#1#2#3{{\setbox0=\hbox{$#1{#2#3}{\int}$ }
\vcenter{\hbox{$#2#3$ }}\kern-.6\wd0}}

\definecolor{purple}{rgb}{0.65, 0, 1}
\definecolor{orange}{rgb}{1,.5,0}
\definecolor{purple}{rgb}{0.5, 0, 0.9}
\definecolor{green}{rgb}{0, 1, 0}
\definecolor{orange}{rgb}{1,.5,0}
\definecolor{gray}{rgb}{.6,.6,.6}

\begin{document}

\title{Remarks on stationary and uniformly-rotating vortex sheets: Flexibility results}

\author{Javier G\'omez-Serrano, Jaemin Park, Jia Shi and Yao Yao}

\begin{abstract}

In this paper, we construct new, uniformly-rotating solutions of the vortex sheet equation bifurcating from circles with constant vorticity amplitude. The proof is accomplished via a Lyapunov-Schmidt reduction and a second order expansion of the reduced system.

\vskip 0.3cm


\end{abstract}

\maketitle

\section{Introduction}

 Let us consider a vortex sheet $\omega$: a weak solution of the 2D Euler equation concentrated on a simple closed curve $\Gamma:=\left\{ z(\theta,t) \in \R^2 : \theta \in \mathbb{T} \right\}$ with vortex-sheet strength $\gamma(\theta,t)$, that  is, for all test functions $\varphi \in C^{\infty}_c(\R^2)$, it holds that
 \begin{align*}
 \int_{\R^2} \varphi(x)d\omega(x,t) = \int_{-\pi}^{\pi} \varphi(z(\theta,t)) \gamma(\theta,t)d\theta.
 \end{align*}
 Note that the evolution of $z$ and $\gamma$ is described by 
 \begin{align}
 \partial_t z(\theta,t) = \mathbf{v}(z,t) + c(\theta,t)\partial_\theta z(\theta,t) \label{evolv_curve}\\
 \partial_t \gamma(\theta,t) = \partial_\theta\left( c(\theta,t)\gamma(\theta,t)\right), \label{evolv_strength}
 \end{align}
 where $\text{curl}(\mathbf{v}) = \omega$ and $c(\theta,t)$ represents the reparametrization freedom of the curve \cite{Castro-Cordoba-Gancedo:naive-vortex-sheet,Lopes-Nussenzveig-Schochet:vortex-sheets-BR-formulation,Majda-Bertozzi:vorticity-incompressible-flow,Sulem-Sulem:finite-time-analyticity-rt}. Recall that the (discontinuous) velocity $\mathbf{v}$ on the curve, generated by the vortex sheet on $\Gamma$ is given by the Birkhoff-Rott integral:
  \begin{align}\label{def_BR}
  \mathbf{v}(z(\theta,t),t) : =BR(z,\gamma)(z(\theta,t)) :=  \frac{1}{2\pi} PV\int_{-\pi}^{\pi} \frac{(z(\theta,t)-z(\eta,t))^{\perp}}{|z(\theta,t) - z(\eta,t)|^2} \gamma(\eta,t) d\eta.
  \end{align}

For simplicity, we will omit the word $PV$ in the notation for principal value integral from now on. We will also denote by $\mathbf{v^{\pm}}$ the respective limits of the velocity on the two sides of $\Gamma$ (with $\mathbf{v}^+$ being the limit on the side that $\mathbf{n}$ points into).
  
The main goal of this paper is to find a class of uniformly-rotating vortex sheets, concentrated on a closed curve which is not a circle. This is the first proof of existence of a family of solutions of such kind. We will say that a vortex sheet is uniformly-rotating with angular velocity $\Omega$ if it is stationary in the rotating frame with angular velocity $\Omega$. See Lemma~\ref{lemma_br_eq} for the exact equations satisfied by a uniformly-rotating vortex sheet. The existence of such solutions is not evident a priori since there are rigidity results by G\'omez-Serrano--Park--Shi--Yao \cite{GomezSerrano-Park-Shi-Yao:rotating-solutions-vortex-sheet-rigidity} ruling out their existence in the case $\Omega \leq 0$ and $\gamma > 0$. 
These solutions are also important since they show that one cannot expect any asymptotic stability of the radial vortex sheet with constant strength. See the recent work by Ionescu--Jia \cite{Ionescu-Jia:axisymmetrization-vortex-point} for an asymptotic stability result when the vorticity is made out of a Dirac delta part and a Gevrey smooth part. Previously, similar solutions (uniformly rotating, non-radial) had been found for the 2D Euler equations in the context of vortex patches or smooth, compactly-supported functions \cite{Castro-Cordoba-GomezSerrano:uniformly-rotating-smooth-euler,Castro-Cordoba-GomezSerrano:analytic-vstates-ellipses,Burbea:motions-vortex-patches,Hmidi-Mateu-Verdera:rotating-vortex-patch,Garcia-Hmidi-Soler:non-uniform-vstates-euler,Hassainia-Masmoudi-Wheeler:global-bifurcation-vortex-patches}.

Despite the complexity of the solutions shown in numerical/actual experiments \cite{Krasny:vortex-sheet-icm,VanDyke:fluid-book,Majda:vortex-sheet-iciam}, there have been significant efforts to prove the existence of solutions to \eqref{evolv_curve}--\eqref{evolv_strength} in various settings. For a $L^2_{loc}$ initial velocity whose vorticity has a definite sign, it turns out that there exists a global weak solution \cite{Delort:vortex-sheet,Majda:vortex-sheet} by the works of Delort, and Majda. In case that the vorticity does not have a definite sign, the existence was proved by Lopes Filho--Nussenzveig Lopes--Xin \cite{Lopes-Nussenzveig-Xin:vortex-sheets-reflection} under the assumption that the initial vorticity satisfies a reflection symmetry. For analytic initial data, local-in-time existence of analytic solutions was proved by Sulem--Sulem--Bardos--Frisch \cite{Sulem-Sulem:finite-time-analyticity-rt}.

 The singularity formation was conjectured by Birkhoff--Fisher and Birkhoff \cite{Birkhoff-Fisher:vortex-sheets-roll,Birkhoff:helmholtz-taylor-instability}. For analytic initial data, the possibility that the curvature may blow up in finite time was supported by asymptotic analysis of Moore \cite{Moore:singularity-vortex-sheet} and also verified by numerical simulations by Krasny and Meiron--Baker--Orszag in \cite{Krasny:singularity-vortex-sheet-point-vortex,Meiron-Baker-Orszag:vortex-sheet}.  Note that the system \eqref{evolv_curve} and \eqref{evolv_strength} is known to be ill-posed in $H^{s}$ for $s>\frac{3}{2}$ \cite{Caflisch-Orellana:singular-solutions-ill-posedness-vortex-sheet}. For more comprehensive discussion on the well-posedness theory, we refer to \cite{Majda-Bertozzi:vorticity-incompressible-flow,Saffman:book-vortex-dynamics,Wu:vortex-sheet}.

\subsection{Steady solutions of the vortex sheet}

There are very few known examples of nontrivial steady solutions, and in fact, other than the circle or the line, the list only comprises the segment of length $2a$ and density 
\begin{equation}\label{example_rotate}
\gamma(x) = \Omega \sqrt{a^2 - x^2}, \qquad x \in [-a,a],
\end{equation}
which is a rotating solution with angular velocity $\Omega$ \cite{Batchelor:book-fluid-dynamics} and the family found by Protas--Sakajo \cite{Protas-Sakajo:rotating-equilibria-vortex-sheet}, made out of segments rotating about a common center of rotation with endpoints at the vertices of a regular polygon. We remark that none of these are supported on a closed curve.

Numerically, O'Neil \cite{ONeil:relative-equilibria-vortex-sheets,ONeil:collapse-vortex-sheets} used point vortices to approximate the vortex sheet and compute uniformly rotating solutions and Elling \cite{Elling:vortex-sheet-cusps} constructed numerically self-similar vortex sheets forming cusps. O'Neil \cite{ONeil:point-vortices-vortex-sheets,ONeil:point-vortices-vortex-sheets-pof} also found numerically steady solutions which are combinations of point vortices and vortex sheets.

\subsection{Main strategy}

The main strategy to prove Theorem \ref{rotatingsolution}, the Main Theorem of this paper, is to employ bifurcation theory and try to bifurcate from the simple eigenvalue $b:= \frac{1}{2\pi}\int_{-\pi}^{\pi} \gamma(x,t)dx =2$. However, the standard methods (Crandall-Rabinowitz \cite{Crandall-Rabinowitz:bifurcation-simple-eigenvalues}) fail since the linearized operator around the circle does not satisfy the transversality condition: in other words, the nontrivial zero set is not transversal to the trivial one (disks with constant vorticity amplitude). This phenomenon is usually known in the literature as a \textit{degenerate bifurcation} \cite{Kielhofer:bifurcation-book,Kielhofer:degenerate-bifurcation}. Graphically, this can be seen in Figure \ref{fig_bifurcation}. The problem is that we no longer have a single branch emanating from the disk, but two, and therefore the linearized operator fails to describe the local behaviour at the bifurcation point. To overcome this issue, we first reduce the nonlinear problem to a suitable finite dimensional space by means of a Lyapunov-Schmidt reduction since the restriction of $D\mathcal{F}$ is an isomorphism between Ker$(D\mathcal{F})^{\perp}$ and Im$(D\mathcal{F})$. After having done so, we are left with a finite dimensional system and it is there where we perform a higher order expansion around the bifurcation point, since, as expected by the failure of the transversality condition, the first order approximation is identically zero. We obtain that in suitable coordinates, the zero sets of $\mathcal{F}$ behave as $x^2 - y^2 = 0$ and thus two bifurcation branches emanate from the bifurcation point. The last part of the proof is devoted to handle the higher order terms, which can be controlled if we restrict the bifurcation domain to a suitable small enough neighbourhood. We mention here that this technique had been successfully employed by Hmidi--Mateu \cite{Hmidi-Mateu:degenerate-bifurcation-vstates-doubly-connected-euler} (in the hyperbolic case) and Hmidi--Renault \cite{Hmidi-Renault:existence-small-loops-doubly-connected-euler} (in the elliptic case).

\subsection{Organization of the paper}

The paper is organized as follows. In Section \ref{sec_eqns} we will write down the equations and describe the main spaces used in the proof. Section \ref{sec_proof} will be devoted to the bulk of the proof of Theorem \ref{rotatingsolution}, with the explicit calculations detailed in Appendix \ref{appendix}. Numerical calculations of the bifurcation branch and a brief discussion of the numerical methods employed are presented in Section \ref{sec_numerics}.

\section{The equations and the functional spaces}\label{sec_eqns}

Let $\omega(\cdot,t)=\omega_0(R_{\Omega t})$ be a stationary/rotating vortex sheet solution to the incompressible 2D Euler equation, where $ \omega_0\in \mathcal{M}(\mathbb{R}^2) \cap H^{-1}(\mathbb{R}^2)$. Here $\Omega=0$ corresponds to a stationary solution, and $\Omega\neq 0$ corresponds to a rotating solution.  Assume that $\omega_0$ is concentrated on $\Gamma$. Throughout this paper we will assume that $\Gamma$ is a simple closed curve and $\Omega > 0$. Following \cite[Lemma 2.1]{GomezSerrano-Park-Shi-Yao:rotating-solutions-vortex-sheet-rigidity}, we have that:

\begin{lemma}\label{lemma_br_eq}
Assume $\omega(\cdot,t)=\omega_0(R_{\Omega t} x)$ is a stationary/uniformly-rotating vortex sheet with angular velocity $\Omega\in\mathbb{R}$, and $\omega_0$ is concentrated on $\Gamma$, with $z$ and $\gamma$ defined as above. Then the Birkhoff-Rott integral $BR$ \eqref{def_BR} and the strength $\gamma$ satisfy the following two equations:
\begin{equation}\label{BR1}
(BR-\Omega x^\perp)\cdot \mathbf{n} =\mathbf{v}^+ \cdot \mathbf{n} = \mathbf{v}^- \cdot \mathbf{n} = 0 \quad\text{ on } \Gamma,
\end{equation}
and
\begin{equation}\label{BR2}
(BR(z(\alpha))-\Omega z^\perp(\alpha)) \cdot \mathbf{s}(z(\alpha)) \,\frac{\gamma(\alpha)}{|z'(\alpha)|}= C.
\end{equation}
\end{lemma}

Note that \eqref{BR2} can be written as
 \begin{align}
 \left( I - P_0 \right) \left[ \left( BR(z,\Gamma)(z(\theta)) - \Omega z(\theta)^{\perp} \right) \cdot \frac{z'(\theta)\gamma(\theta)}{|z'(\theta)|^2} \right]= 0, \label{noevol_strength}
 \end{align}
 where $P_0$ is a projection to the mean, that is, $P_0f := \frac{1}{2\pi} \int_{-\pi}^{\pi} f(\theta)d\theta$. For simplicity, we  also denote $\fint f(\theta)d\theta := \frac{1}{2\pi}\int_{-\pi}^{\pi} f(\theta)d\theta$. Now plugging $z(\theta) = (1+r(\theta))(\cos(\theta),\sin(\theta))$ and $\gamma(\theta) := b + g(\theta)$ into \eqref{BR1}, \eqref{BR2} and \eqref{noevol_strength} yields that
 \begin{align}\label{mainfunctional}
 \mathcal{F}(b,g,r) := (\mathcal{F}_1,\mathcal{F}_2) = (0,0),
 \end{align}
 where
 \begin{align*}
 &\mathcal{F}_1(b,g,r) := \fint_{-\pi}^{\pi}(b+g(\eta))\frac{\left(r'(\theta)\cos(\theta-\eta)-(1+r(\theta))\sin(\theta-\eta)\right)(1+r(\eta))-(1+r(\theta))r'(\theta)}{(1+r(\theta))^2+(1+r(\eta))^2-2(1+r(\theta))(1+r(\eta))\cos(\theta-\eta)}d\eta\\
& \qquad\qquad\qquad + \Omega r'(\theta)(1+r(\theta)), \\
&\mathcal{F}_2 (b,g,r) := \left(I-P_0 \right) \tilde{\mathcal{F}}_2(b,g,r), \\
&\tilde{\mathcal{F}}_2(b,g,r) := \fint_{-\pi}^{\pi}(b+g(\eta))\frac{(1+r(\theta))^2-(r'(\theta)\sin(\theta-\eta)+(1+r(\theta))\cos(\theta-\eta))(1+r(\eta))}{(1+r(\theta))^2+(1+r(\eta))^2-2(1+r(\theta))(1+r(\eta))\cos(\theta-\eta)}d\eta \\
&\qquad\qquad\quad\quad \times \frac{(b+g(\theta))}{r'(\theta)^2+(1+r(\theta))^2}  -\Omega (1+r(\theta))^2 \frac{b+g(\theta)}{r'(\theta)^2+(1+r(\theta))^2}.
 \end{align*}

  Throughout the paper we will work with the following analytic function spaces. Let $c>0$ be a sufficiently small parameter and let $\mathcal{C}_{w}(c)$ be the space of analytic functions in the strip $|\Im(z)| \leq c$. For $k \in \mathbb{N}$, denote
\begin{align*}
X^{k}_{c} := \left\{ f(\theta) \in \mathcal{C}_{w}(c), \quad f(\theta) = \sum_{n=1}^{\infty}a_n\cos(2n\theta), \quad \sum_{\pm} \int_{-\pi}^{\pi}  |f(\theta \pm ic)|^2 + |\partial^k f(\theta \pm ic)|^2 d\theta< \infty \right\} \\
\quad Y^{k}_{c} := \left\{f(\theta) \in \mathcal{C}_{w}(c), \quad f(\theta) = \sum_{n=1}^{\infty}a_n\sin(2n\theta),\quad \sum_{\pm} \int_{-\pi}^{\pi}  |f(\theta \pm ic)|^2 + |\partial^k f(\theta \pm ic)|^2 d\theta< \infty \right\},
\end{align*}

  From now on, due to scaling considerations, we will fix $\Omega = 1$ and $b$ will play the role as bifurcation parameter. It is clear that $\mathcal{F}(b,0,0) = (0,0)$ for all $b\in \R$ since $\mathcal{F}_1(b,0,0) = 0$ and $\tilde{\mathcal{F}}_2(b,0,0)$ is constant. Our main theorem in this paper is the following:   
  \begin{theorem}\label{rotatingsolution}
Let $k \geq 3$, and let $c>0$ be sufficiently small. Then, there exists a curve of solutions $(b,g,r)$ of $\mathcal{F} = (0,0)$, belonging to $\R\times X^{k}_{c}\times X^{k+1}_{c}$ and a neighbourhood of $(b,g,r) = (2,0,0)$, bifurcating from $(b,g,r) = (2,0,0)$ such that $(g,r) \neq (0,0)$.
  \end{theorem}

\section{Proof of the Main Theorem}\label{sec_proof}

The goal of this section is to prove the existence of non-radial uniformly-rotating vortex sheets.  To do so, we will split the proof into the following steps: first we will prove that the functional $\mathcal{F}$ is $C^3$, next we will study $D\mathcal{F}$ to show that, as mentioned in the introduction, it is a Fredholm operator of index 0, with dim(Ker$(D\mathcal{F})) = 1$. The next step is to apply Lyapunov-Schmidt theory and reduce the problem to a finite (2) dimensional one. In those coordinates, linear expansions fail to be conclusive (all the linear terms vanish) since 2 nontrivial branches emanate from the bifurcation point (as opposed to 1). Instead, we perform a quadratic expansion to determine that locally the bifurcation branches look like two pairs of straight lines (specifically as $x^2-y^2 = 0$ in some well-chosen coordinates) and hence the bifurcation does not trivialize (as if it had been of the type $x^2 + y^2 = 0$). We conclude the proof by handling the higher order terms and showing that they don't alter the quadratic behaviour in a sufficiently small neighbourhood of the bifurcation point.

\subsection{Continuity of the functional}
In this subsection, we will check the regularity of $\mathcal{F}$. As explained above, we will reduce the infinite dimensional problem to a finite dimensional problem and investigate its Taylor expansion up to quadratic order. Hence, we need to check if the functional is regular enough to do so. To this end, we have the following proposition:

\begin{proposition}\label{regularity1}
Let $k\ge 3$. Then there exists a neighborhood $U$ of $(2,0,0) \in \R\times X^{k}_{c}\times X^{k+1}_{c} $ such that $\mathcal{F} \in C^{3}\left( U ; Y^{k}_{c}\times X^{k}_{c} \right)$. 
\end{proposition}
\begin{proof}
Since the stream function, $\omega * \mathcal{N}$, is invariant under rotations, it follows immediately that $\mathcal{F}$ is also invariant under rotation by $\pi$-radians, hence $\mathcal{F}$ has only even Fourier modes. Also the oddness of $\mathcal{F}_1$ and evenness of $\mathcal{F}_2$ follow from the invariance under reflection.
 
 To prove the regularity, we briefly sketch the idea. We impose $k\ge 3$ to ensure that $H^{k}$ is a Banach algebra. It is clear that $\mathcal{F}$ is smooth in $b$. It is also straightforward that, for example, for all $(g,r)$ near $(0,0)\in X^{k}_{c}\times X^{k+1}_{c}$, 
 \begin{align*}
\hspace*{-0.5cm} \frac{d}{dt}\mathcal{F}_1&(b,g+th_1,r+th_2)\bigg|_{t=0} = PV \fint h_1(\eta)\frac{\left(r'(\theta)\cos(\theta-\eta)-(1+r(\theta))\sin(\theta-\eta)\right)(1+r(\eta))-(1+r(\theta))r'(\theta)}{(1+r(\theta))^2+(1+r(\eta))^2-2(1+r(\theta))(1+r(\eta))\cos(\theta-\eta)} \\
 &\quad + (b+g(\eta))\left[ \frac{(h_2'(\theta)\cos(\theta-\eta)-h_2(\theta)\sin(\theta-\eta))(1+r(\eta))-h_2(\theta)r'(\theta)}{(1+r(\theta))^2+(1+r(\eta))^2-2(1+r(\theta))(1+r(\eta))\cos(\theta-\eta)} \right.\\
 &\quad \left. + \frac{\left(r'(\theta)\cos(\theta-\eta)-(1+r(\theta))\sin(\theta-\eta)\right)h_2(\eta)-(1+r(\theta))h_2'(\theta)}{(1+r(\theta))^2+(1+r(\eta))^2-2(1+r(\theta))(1+r(\eta))\cos(\theta-\eta)} \right. \\
 &\quad \left. -\left[\left( r'(\theta)\cos(\theta-\eta)-(1+r(\theta))\sin(\theta-\eta)\right)(1+r(\eta))-(1+r(\theta))r'(\theta) \right] \right. \\
 &\quad \left. \times \frac{\left[ 2(1+r(\theta)h_2(\theta)+2(1+r(\eta)h_2(\eta)-2\cos(\theta-\eta)(h_2(\theta)(1+r(\eta))+h_2(\eta)(1+r(\theta)))\right]}{((1+r(\theta))^2+(1+r(\eta))^2-2(1+r(\theta))(1+r(\eta))\cos(\theta-\eta))^2} \right] d\eta + \Omega h_2'(\theta)\\
& =: D\mathcal{F}_1(b,g,r)[h_1,h_2],
 \end{align*}
 and $D\mathcal{F}_1:\R\times X^{k}_{c}\times X^{k+1}_{c} \mapsto \mathcal{L}(X^{k}_{c}\times X^{k+1}_{c} ; Y^{k}_{c}\times X^{k}_{c})$ is continuous. A similar derivation can be performed for $D\mathcal{F}_2$.  For the higher derivatives, we refer to \cite{Castro-Cordoba-GomezSerrano:existence-regularity-vstates-gsqg,Castro-Cordoba-GomezSerrano:analytic-vstates-ellipses,GomezSerrano:stationary-patches,Hmidi-Mateu-Verdera:rotating-vortex-patch,Renault:relative-equlibria-holes-sqg} for the method to deal with the singular integrals arising throughout the calculations.
\end{proof}

\subsection{Fredholm index of the linearized operator $D\mathcal{F}$}

This subsection is devoted to show that  $D\mathcal{F}$ is Fredholm of index zero. We can make all the calculations explicit, moreover the operator diagonalizes in Fourier modes. We have the following lemmas:

\begin{lemma}\label{linearized}
Let $g(\theta)=\sum_{n=1}^{\infty}a_n \cos(2n\theta)$ and $r(\theta) = \sum_{n=1}^{\infty}b_n\cos(2n\theta)$. Then we have that
\[
D\mathcal{F}(b,0,0)\left[ g, r\right] = \left(\begin{array}{c}\hat{g}(\theta) \\ \hat{r}(\theta) \end{array}\right),
\]
where
\[ \hat{g}(\theta) = \sum_{n=1}^{\infty} \hat{a}_n \sin(2n\theta), \quad \hat{r}(\theta) = \sum_{n=1}^{\infty} \hat{b}_n \cos(2n\theta),
\]
and the coefficients satisfy, for any $n \ge 1$:
\[
M_n
\begin{pmatrix}
a_n \\
b_n
\end{pmatrix}
:=
\begin{pmatrix}
-\frac{1}{2} & -2n\left( \Omega - \frac{b}{2}\right) \\
\frac{b}{2}-\Omega & b^2(n-1)
\end{pmatrix}
\begin{pmatrix}
a_n \\
b_n
\end{pmatrix}
=
\begin{pmatrix}
\hat{a}_n \\
\hat{b}_n
\end{pmatrix}.
\]
\end{lemma}
\begin{proof}
 We use \eqref{linearf1} in Lemma~\ref{derivative} and obtain
\begin{align*}
\hat{g}(\theta) = \frac{d}{dt}\mathcal{F}_1(b,tg,tr) &=  - \fint \frac{g(\eta)\sin(\theta-\eta)}{2-2\cos(\theta-\eta)}d\eta + \left( \Omega- \frac{b}{2}\right) r'(\theta)\\
& = -\sum_{n=1}^{\infty}a_n\fint \frac{\cos(2n\eta)\sin(\theta-\eta)}{2-2\cos(\theta-\eta)}d\eta + \sum_{n=1}^{\infty}(-2n)\left( \Omega - \frac{b}{2} \right)b_n \sin(2n\theta) \\
& = \sum_{n=1}^{\infty}\left ( -\frac{a_n}{2}  + (-2n)\left( \Omega - \frac{b}{2}\right) b_n \right)\sin(2n\theta),
\end{align*}
where the last equality follows from ~\eqref{lemma2}. Similarly, we apply \eqref{linearf2} in Lemma~\ref{derivative} and \eqref{lemma1} to obtain
\begin{align*}
\hat{r}(\theta) & = \left( \frac{b}{2} - \Omega \right) g(\theta) + b^2\left( \fint \frac{r(\theta)-r(\eta)}{2-2\cos(\theta-\eta)}d\eta - r(\theta)\right)\\
& = \sum_{n=1}^{\infty} \left(\frac{b}{2} - \Omega \right) a_n\cos(2n\theta) + b^2 \sum_{n=1}^{\infty} b_n \left(  \fint \frac{\cos(2n\theta)-\cos(2n\eta)}{2-2\cos(\theta-\eta)}d\eta - \cos(2n\theta) \right)\\
& = \sum_{n=1}^{\infty} \left( \left(\frac{b}{2} - \Omega \right) a_n+ b^2\left( n - 1\right)b_n \right)\cos(2n\theta).
\end{align*}
This proves the lemma.
\end{proof}

\begin{lemma}\label{onedimensionality}
Let us fix $b=2$ and $\Omega = 1$. We also denote $v:=(0,\cos(2\theta)) \in X^k_{c}\times X^{k+1}_{c}$ and $w := (0,\cos(2\theta)) \in Y^k_{c} \times X^k_{c}$. Then it holds that 
\begin{align*}
&\text{Ker}\left(D\mathcal{F}(2,0,0)\right) = \text{span}\left\{ v \right\} \subset{X^k_{c}\times X^{k+1}_{c}},\\
& \text{Im}\left(D\mathcal{F}(2,0,0)\right)^{\perp} = \text{span}\left\{ w \right\}  \subset Y^{k}_{c}\times X^{k}_{c}.
\end{align*}
\end{lemma}
\begin{proof}
From Lemma~\ref{linearized}, we have
\begin{align*}
M_n =
\begin{pmatrix}
-\frac{1}{2} & 0 \\
0  & 4(n-1)
\end{pmatrix},
\end{align*}
for all $n \ge 1$.
For all $n\ge 2$, $M_n$ is clearly an isomorphism, while $\text{Ker}(M_1) = \text{Im}(M_1)^{\perp} = \left(\begin{array}{c} 0  \\ 1 \end{array}\right)$. By orthogonality of Fourier modes, this proves the lemma.
\end{proof}

 \subsection{Lyapunov-Schmidt reduction}\label{lsreduction}
In this subsection, we will aim to derive a finite dimensional system which is equivalent to \eqref{mainfunctional}. From Lemma~\ref{onedimensionality}, we have the following orthogonal decompositions of the function spaces:
\begin{align*}
 X := X^k_{c}\times X^{k+1}_{c} = \text{span}\left\{ v \right\} \oplus \text{Ker}\left(D\mathcal{F}(2,0,0)\right)^{\perp} =: \text{span}\left\{ v \right\} \oplus \mathcal{X}, \quad v\in \text{Ker}\left(D\mathcal{F}(2,0,0)\right),\\
 Y := Y^{k}_{c} \times X^{k}_{c}= \text{span}\left\{ w \right\} \oplus \text{Im}\left(D\mathcal{F}(2,0,0)\right) =: \text{span}\left\{ w \right\} \oplus \mathcal{Y},  \quad w\in \text{Im}\left(D\mathcal{F}(2,0,0)\right)^{\perp},
\end{align*}
where $v$ and $w$ are as defined in Lemma~\ref{onedimensionality}. Let us consider the orthogonal projections 
\[
 P : X \rightarrow \text{span}\left\{ v \right\},\quad 
 Q: Y \rightarrow \text{span}\left\{ w \right\}. 
\]
 More precisely, we have
 \begin{align}
& P(g(\theta),r(\theta)) = \left( 0 , \left(\frac{1}{\pi}\int r(\eta)\cos(2\eta)d\eta\right) \cos(2\theta) \right) \quad \text{ for all }(g,r)\in X^{k}_{c}\times X^{k+1}_{c}, \label{defofpq}\\
 &Q(G(\theta),R(\theta)) =  \left( 0 , \left(\frac{1}{\pi}\int R(\eta)\cos(2\eta)d\eta\right) \cos(2\theta) \right) \quad \text{ for all } (G,R) \in Y^{k}_{c}\times X^{k}_{c}. \label{defofpq1}
 \end{align}
  We remark that we will sometimes abuse notation and identify $\mathcal{F}(b,g,r)$ with $\mathcal{F}(b,(g,r))$, where $(g,r) \in X$. Let us define $G:\R \times \text{span}\left\{ v \right\} \times \mathcal{X} \mapsto {Y}$ as follows: 
  \begin{align*}
  G(b,f,x):=\mathcal{F}(b,f+x), \quad\text{ for }\quad b\in \R,\quad  f\in \text{span}\left\{ v \right\},\quad x\in \mathcal{X}.
  \end{align*}
  Then \eqref{mainfunctional} is equivalent to (for $(g,r) = f+x$)
  \begin{align}\label{Reduction1}
  QG(b,f,x)=0 \quad \text{ and }\quad \left( I-Q\right)G(b,f,x)=0.
  \end{align}
 However, it follows from Lemma~\ref{onedimensionality} that 
 \begin{align}\label{isomorphism}
 D_x \left((I-Q)G\right)(b,0,0)=\left(I-Q\right)D\mathcal{F}(b,0)P : \mathcal{X}\mapsto \mathcal{Y}
 \end{align}
  is an isomorphism, consequently, the implicit function theorem yields that there is an open set $U\subset \R\times \text{span}\left\{ v \right\}$ near $(b,0)$ and a function $\varphi:U\mapsto \mathcal{X}$ such that 
  \begin{align*}
  \left( I-Q\right)G(b,f,\varphi(b,f)) = \left(I-Q\right)\mathcal{F}(b,f+\varphi(b,f))=0.
  \end{align*}
 Note that from $\mathcal{F}(b,0)=(0,0)$ for any $b\in \R$, we have 
 \begin{align}
 \label{trivial_solution}
 \varphi(b,0)=0,
 \end{align}
 and thus \eqref{Reduction1} is equivalent to
 \begin{align}\label{Reduction2}
0 = QG(b,f,\varphi(b,f)) = Q\mathcal{F}(b,f+\varphi(b,f)), \quad (b,f)\in U.
 \end{align}
 Since $\text{span}\left\{ v \right\}$ is one dimensional, we have $f=tv$ for some $t\in \R$, therefore the system \eqref{Reduction2} can be written in terms of the variables $b$ and $t$  as
 \begin{align*}
 0 = Q\mathcal{F}(b,tv+\varphi(b,tv)) &= \int_{0}^{1}\frac{d}{d s}\left( Q\mathcal{F}(b,stv+\varphi(b,stv)) \right)ds\\
 &= \int_0^1 QD\mathcal{F}(b,stv+\varphi(b,stv))(tv+t\partial_f\varphi(b,stv)v) ds,
 \end{align*}
 where we used \eqref{trivial_solution} to obtain the second equality.
 Dividing the right-hand side by $t$ to get rid of the trivial solutions, we are led to solve the following two dimensional problem:
 \begin{align}\label{Reduction4}
0 = F_{red}(b,t):=\int_0^1 Q D\mathcal{F}(b,stv+\varphi(b,stv))(v+\partial_f\varphi(b,stv)v) ds, \quad (b,tv)\in U.
 \end{align}

\subsection{Quadratic expansion of the reduced functional}
The main idea is to expand the reduced functional $F_{red}$ up to quadratic terms. To this end, we recall the following proposition for the derivatives of $F_{red}$.

\begin{proposition}(\cite[Proposition 3]{Hmidi-Mateu:degenerate-bifurcation-vstates-doubly-connected-euler}, \cite[Proposition 3.1]{Hmidi-Renault:existence-small-loops-doubly-connected-euler})
Let $F_{red}$ be defined as in \eqref{Reduction4}. Then the following hold:
\begin{itemize}
\item[(a)] First derivatives:
\begin{align*}
&\partial_b F_{red}(2,0) = Q\partial_b D\mathcal{F}(2,0)v, \\
&\partial_t F_{red}(2,0) = \frac{1}{2}\frac{d^2}{dt^2}Q\mathcal{F}(2,tv)\bigg|_{t=0}.
\end{align*}
\item[(b)] Second derivatives:
\begin{align*}
&\partial_{bb}F_{red}(2,0) = 2Q\partial_bD\mathcal{F}(2,0)\tilde{v},\\
&\partial_{tt} F_{red}(2,0) = \frac{1}{3}\frac{d^3}{dt^3}\left[ Q\mathcal{F}(2,tv)\right] \bigg|_{t=0} + Q\frac{d^2}{dtds}\mathcal{F}(2,tv+s\hat{v})\bigg|_{t=s=0},\\
&\partial_{tb}F_{red}(2,0) = \frac{1}{2}\partial_b Q\frac{d^2}{dt^2} \mathcal{F}(b,tv) \bigg|_{b=2,t=0} + \frac{1}{2}Q\partial_bD\mathcal{F}(2,0)\hat{v} + Q\frac{d^2}{dtds}\mathcal{F}(2,tv+s\tilde{v})\bigg|_{t=s=0},
\end{align*}
where 
\begin{align*}
&\hat{v} := -\left[ D\mathcal{F}(2,0) \right]^{-1} \frac{d^2}{dt^2}\left[ (I - Q) \mathcal{F}(2,tv) \right] \bigg|_{t=0}, \\
&\tilde{v} := - \left[ D \mathcal{F}(2,0) \right]^{-1} (I - Q)\partial_b D \mathcal{F}(2,0)v.
\end{align*}
\end{itemize}
\end{proposition}

Now using the values found in Lemma~\ref{somevalues}, we can obtain the derivatives of $F_{red}$.

\begin{proposition}\label{explicitvalues}
Let $F_{red}$ be defined as in \eqref{Reduction4}. Then it holds that
\begin{align}
&\partial_b F_{red}(2,0) = 0, \label{derivative11}\\
&\partial_t F_{red}(2,0) = 0. \label{derivative12}\\
&\partial_{bb}F_{red}(2,0) = 2w,\label{derivative21}\\
&\partial_{tt} F_{red}(2,0) = -8w, \label{derivative22}\\
&\partial_{tb}F_{red}(2,0) = 0. \label{derivative23}
\end{align}
\end{proposition}

\begin{proof}
  \eqref{derivative11} follows immediately from \eqref{defofpq1} and \eqref{value1}. For \eqref{derivative12}, we use \eqref{value2} and the orthogonality of the Fourier modes. \eqref{derivative21} follows from \eqref{value11}. \eqref{derivative22} follows from \eqref{value5} and \eqref{value6}. Lastly, \eqref{derivative23} follows from \eqref{value10}, \eqref{value8} and \eqref{value9}.
\end{proof}
\subsection{Proof of Theorem~\ref{rotatingsolution}}
Now we are ready to  prove the main theorem of this section.

\begin{proof}
 From \eqref{Reduction4}, it suffices to show that there exist $(b,t)$ such that $t\ne 0$  and $F_{red}(b,t) = 0$. To do so, we expand $F_{red}$ up to quadratic terms and  obtain that for all $(b,t)$ near $(2,0)$, 
\begin{align*}
F_{red}(b,t) &= \left[F_{red}(2,0) + \partial_{b}F_{red}(2,0)(b-2) + \partial_{t}F_{red}(2,0)t \right. \\
&\left.  \ + \frac{1}{2}\partial_{bb}F_{red}(2,0)(b-2)^2 + \frac{1}{2}\partial_{tt}F_{red}(2,0)t^2 + \partial_{tb}F_{red}(2,0)(b-2)t + \left((b-2)^2 + t^2 \right)\epsilon(b,t) \right]w,
\end{align*}
where $\epsilon(b,t)$ is a continuous function such that $\lim_{(b,t)\to (2,0)}\epsilon(b,t) =\epsilon(2,0)= 0$. From Proposition~\ref{explicitvalues}, it follows that (we drop $w$ for simplicity)
\begin{align*}
F_{red}(b,t) = (b-2)^2 - 4t^2 + \left( (b-2)^2 + t^2 \right)\epsilon(b,t).
\end{align*}
Now we use the change of variables $b:= x + 2$ and $t=xy$, so that 
\begin{align}\label{changeofvariables}
\hat{F}(x,y) := \frac{F_{red}(x+2,xy)}{x^2} = \left(1 - 4y^2\right) + (1+y^2)\epsilon(x+2,xy).
\end{align} 
Clearly, $\hat{F}\left( 0,\frac{1}{2} \right) = 0$ and $ \partial_{y}\hat{F}\left( 0,\frac{1}{2} \right) = -4\ne 0$.
Therefore the implicit function theorem implies that there exists a continuous function $\phi$ near $0$ such that $\hat{F}(x,\phi(x)) = 0$ and $\phi(0) = \frac{1}{2}$. Therefore it follows from \eqref{changeofvariables} that there exists a pair $(b,t)$ such that $t\ne 0$ and $F_{red}(b,t) = 0$. This finishes the proof.
\end{proof}

\begin{figure}[h!]
\begin{center}
\includegraphics[scale=0.2]{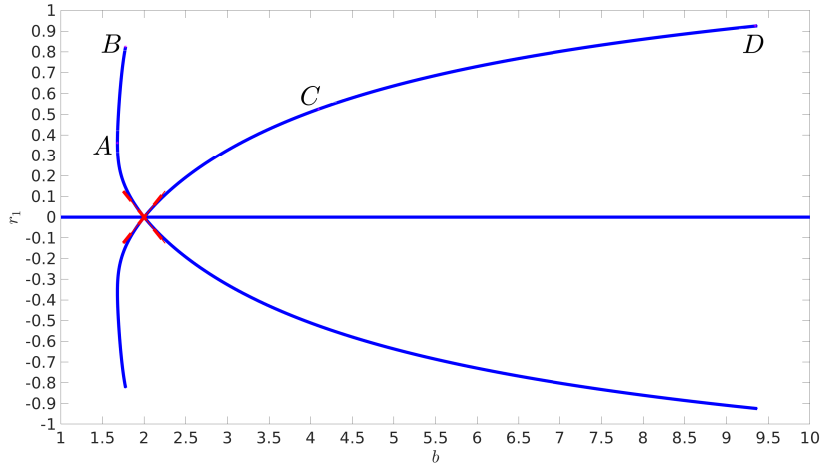}
\caption{\label{fig_bifurcation} Plot of the bifurcation diagram of the solutions given by Theorem \ref{rotatingsolution}. The dotted red lines correspond to the linear expansion \eqref{changeofvariables} around the bifurcation point $(2,0)$. See Figure \ref{fig_panels} for a numerical plot of the solutions at the points $A$, $B$, $C$, $D$. The branches continue beyond what is calculated.}
\end{center}

\end{figure}

\section{Numerical results}\label{sec_numerics}

In this section, we describe how to compute numerically the branches of solutions emanating from the disk, previously proved (locally) in Theorem \ref{rotatingsolution}. See Figure \ref{fig_bifurcation}. To do so, we calculate solutions of the form
\[R(\theta) = 1 + \sum_{k=1}^{N} r_k \cos(2k\theta), \quad \gamma(\theta) = \sum_{k=0}^{N} \gamma_k \cos(2k\theta)
\]
with $\gamma_0 = b$. We first employed continuation in $b$, in increments of $\Delta b = 0.001$, starting from $b=1.8$ and $b=2.1$ and using as initial guess for the starting $b$ the solution given by the linear theory and for the subsequent $b$ the solution found in the previous iteration. After discovering a fold at approximately $b \sim 1.68$, we switched variables and instead we recalculated using continuation in $r_1$, which appears to be monotonic along the branches. As before, we start at $r_1 = \pm 0.125$ and take an increment $\Delta r_1 = 0.001$.

To compute a solution for a fixed $r_1$ we use the Levenberg-Marquardt algorithm. We aim to find a zero of the system of equations $\mathcal{F}(b,g,r)(\theta_j)$, with $\theta_j = \frac{j\pi}{N_\theta}, \quad j=1,\ldots,N_\theta$ and $N_\theta = 1024$ with variables $r_k, \, k \neq 1$  (recall that $r_1$ is fixed at each iteration since it is the continuation parameter) and $\gamma_k$. We take $N = 160$. In order to perform the integration in space, we desingularize the principal value at $\eta = \theta$ by subtracting $\frac12\mathcal{H}(\gamma)$ to $\mathcal{F}_1$, where $\mathcal{H}$ denotes the Hilbert transform, computed explicitly since we have the Fourier expansion of $\gamma$, and perform a trapezoidal integration on the rest (for which the integrand is smooth), with step $h = \frac{2\pi}{N_\theta}$. We remark that the integrand of $\mathcal{F}_2$ has a removable singularity (thus no principal value integration is needed) and can be integrated using the trapezoidal integration if the limit at $\eta = \theta$ is taken properly.

\begin{figure}[h!]
\hspace*{-1.2cm}\begin{tabular}{cc}
 \includegraphics[width=0.25\textwidth]{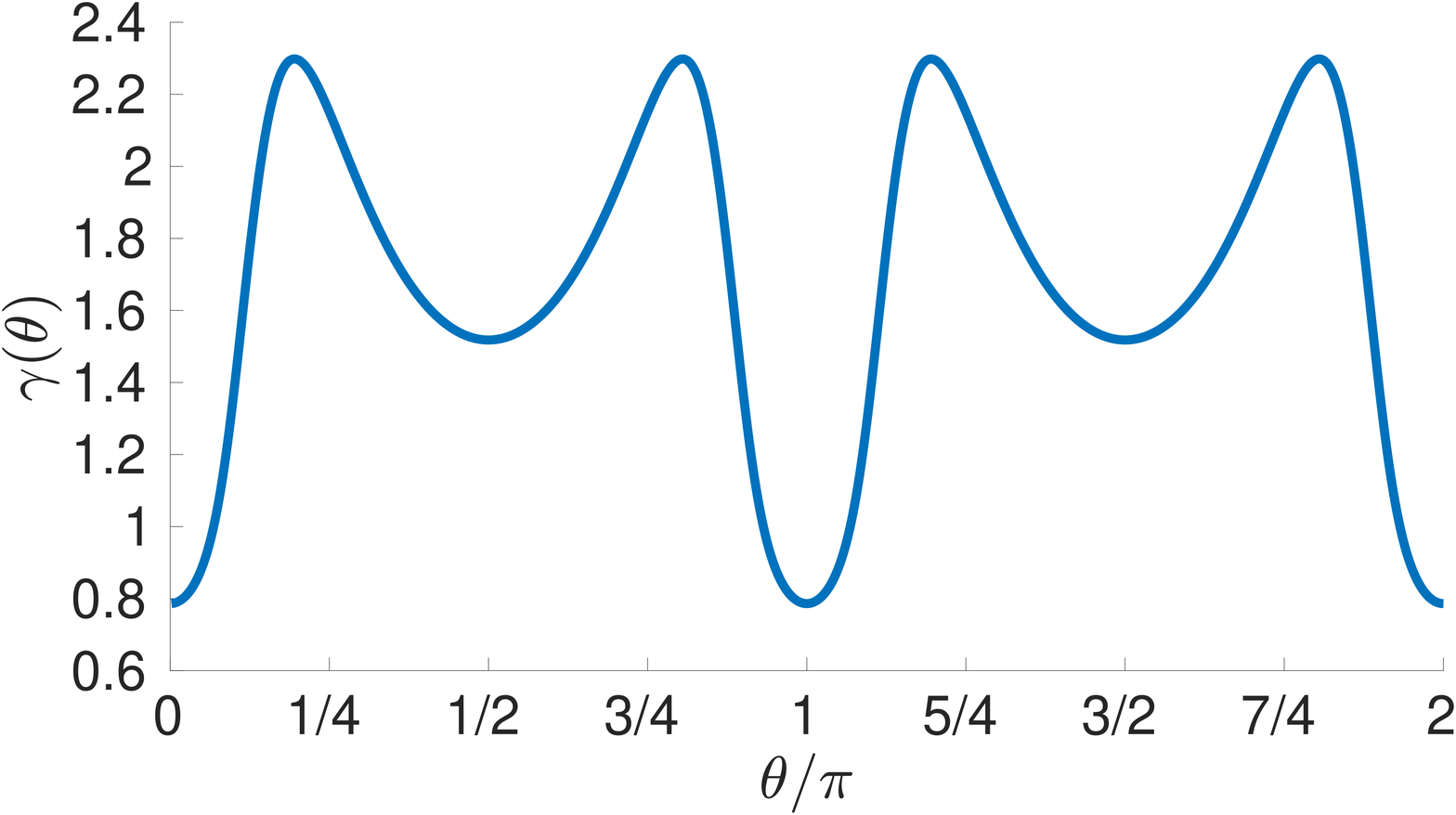}
 \includegraphics[width=0.25\textwidth]{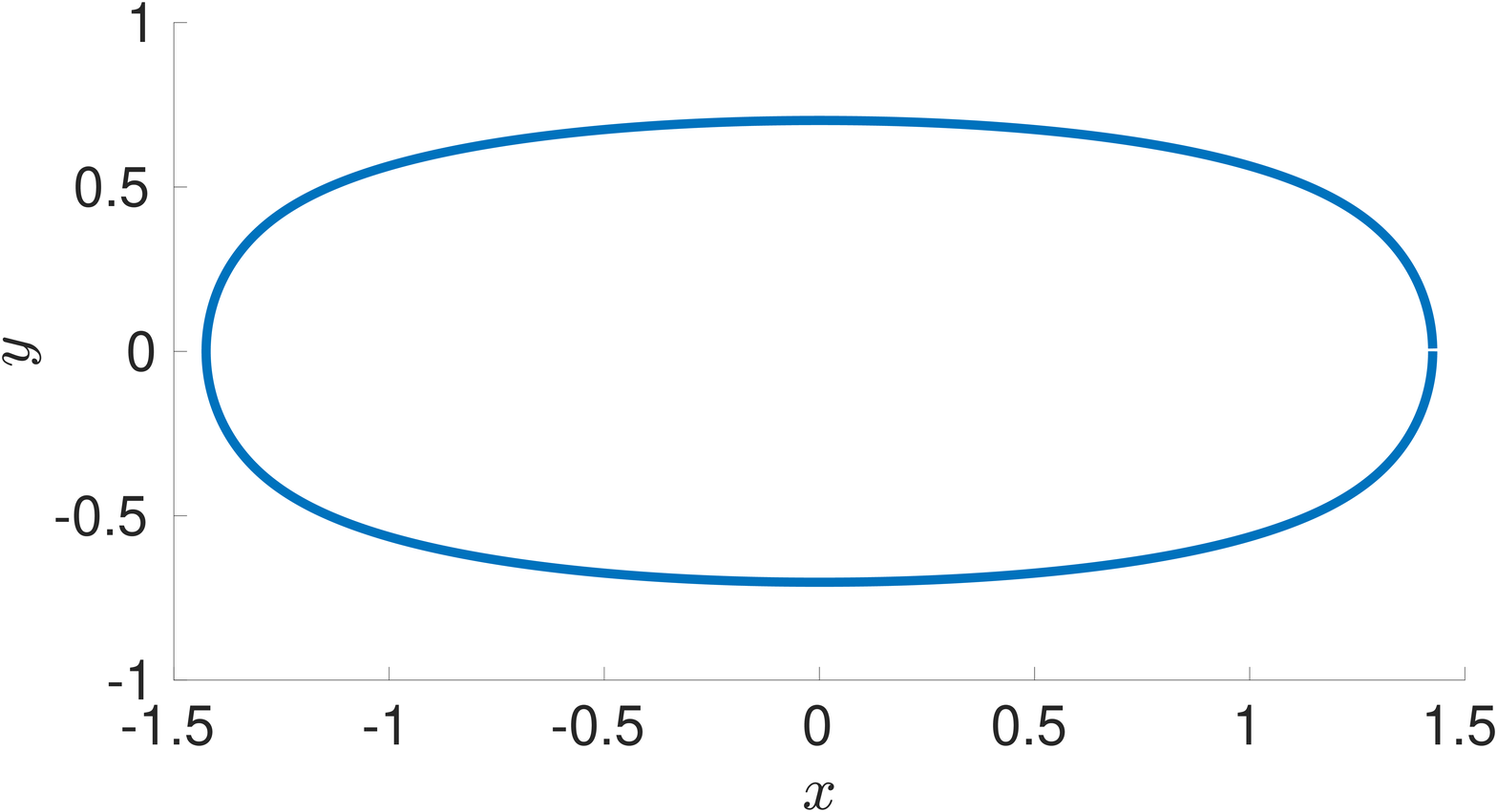}&
 \includegraphics[width=0.25\textwidth]{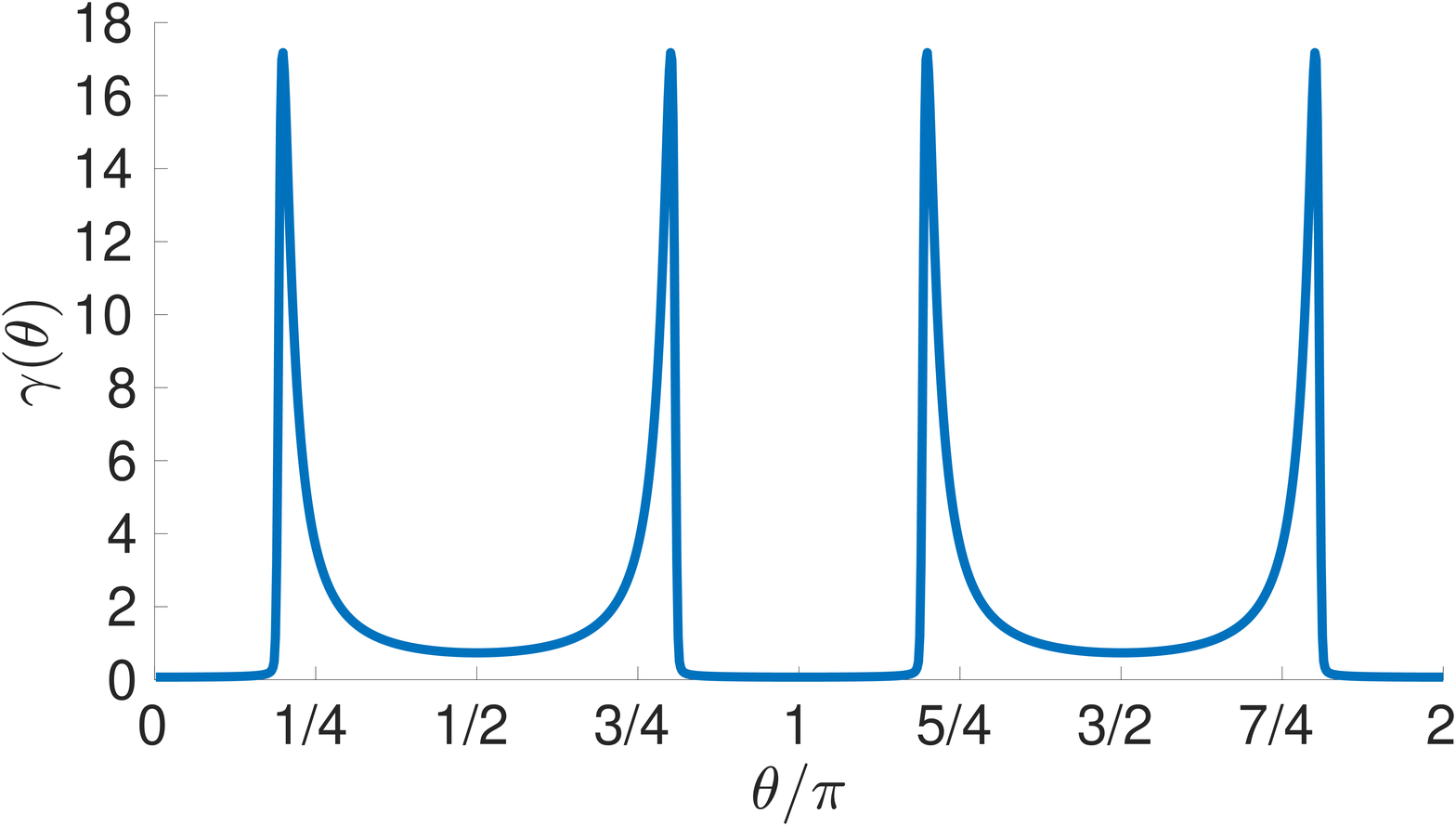}
 \includegraphics[width=0.25\textwidth]{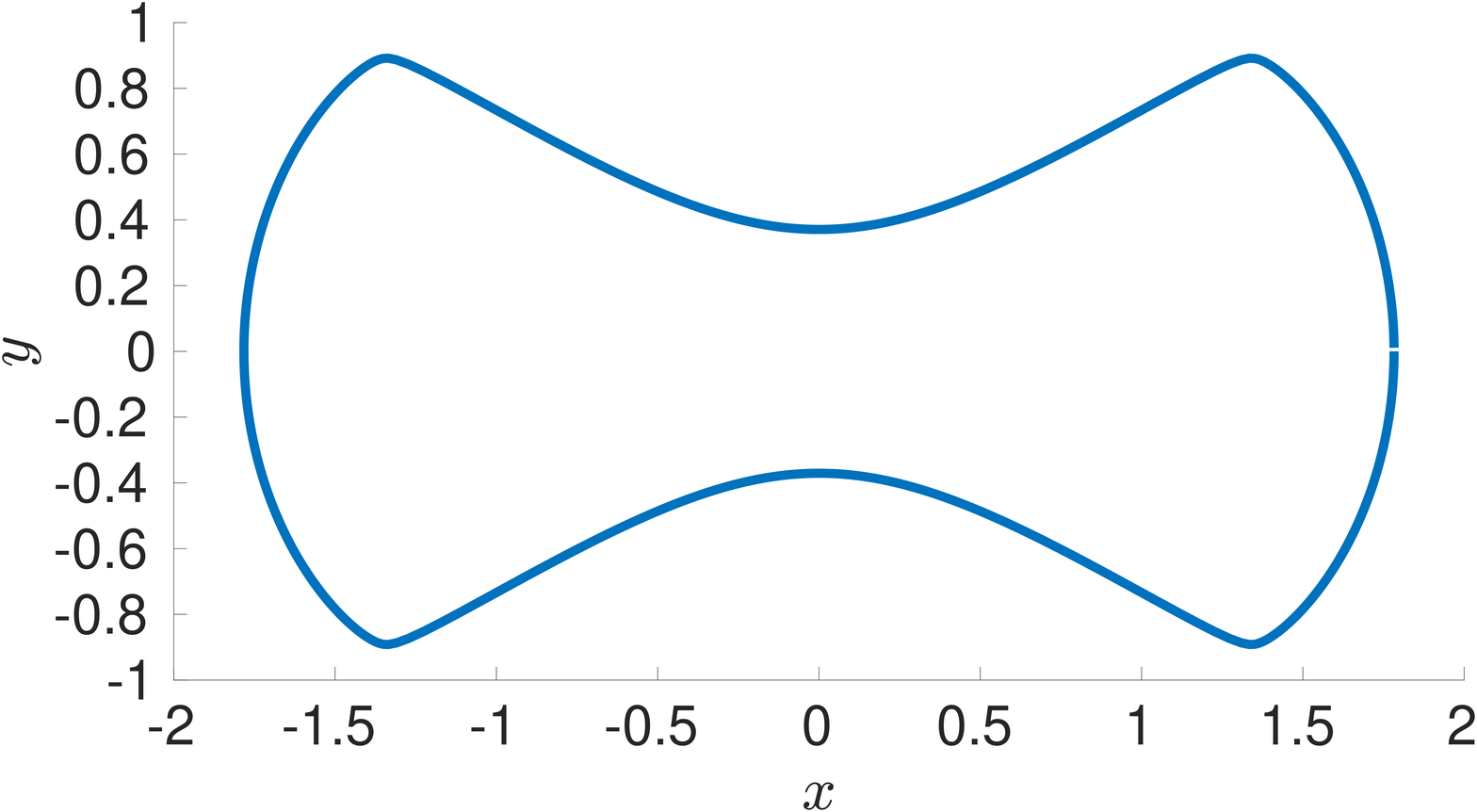}  \\
\qquad \qquad (a) $\gamma(\theta)$ and $z(\theta)$ for $r_1 = 0.362, b \sim 1.6799$ & \qquad (b) $\gamma(\theta)$ and $z(\theta)$ for $r_1 = 0.825, b \sim 1.7779$ \hspace{2.5cm} \\
 \includegraphics[width=0.25\textwidth]{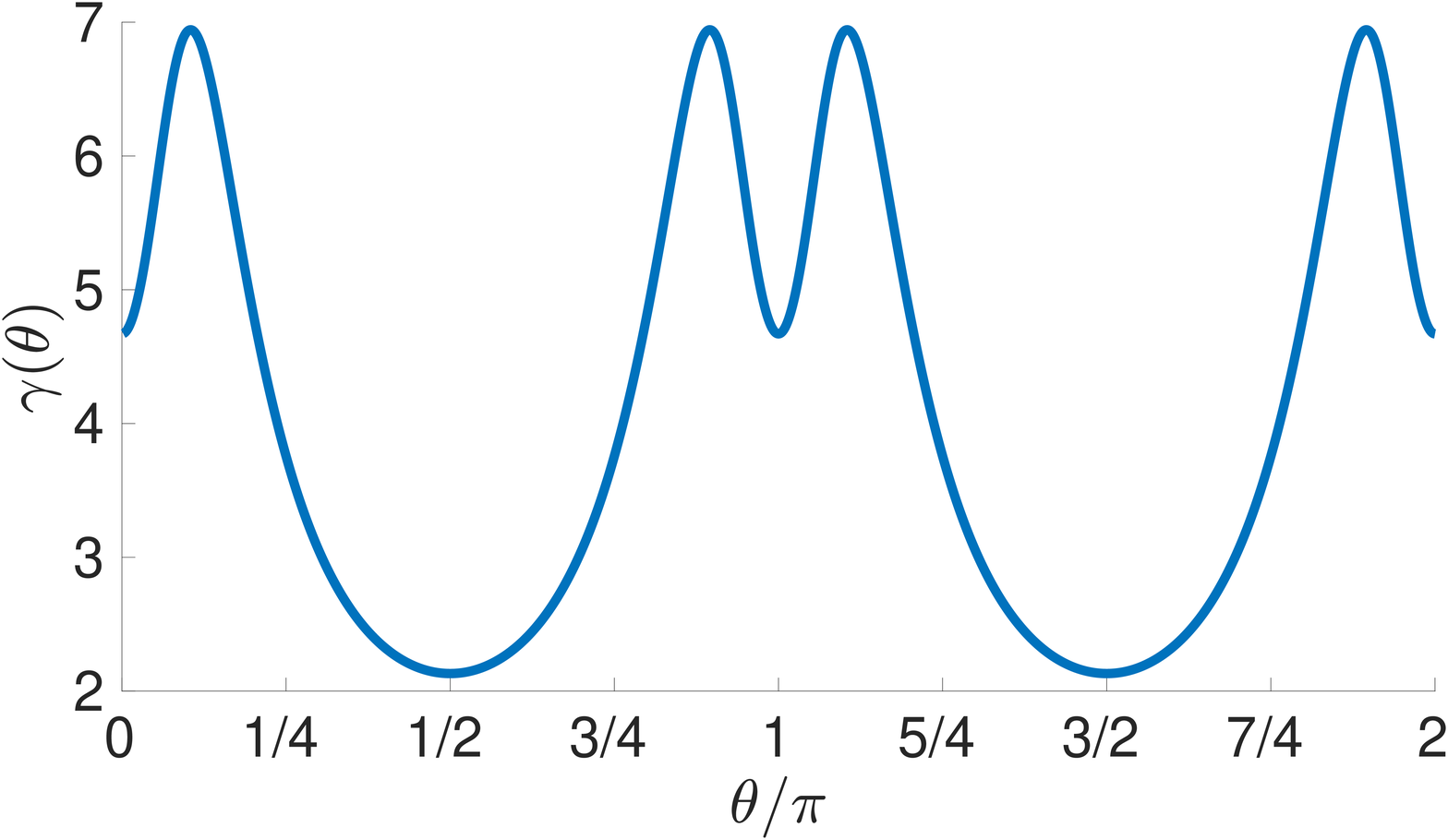}
 \includegraphics[width=0.25\textwidth]{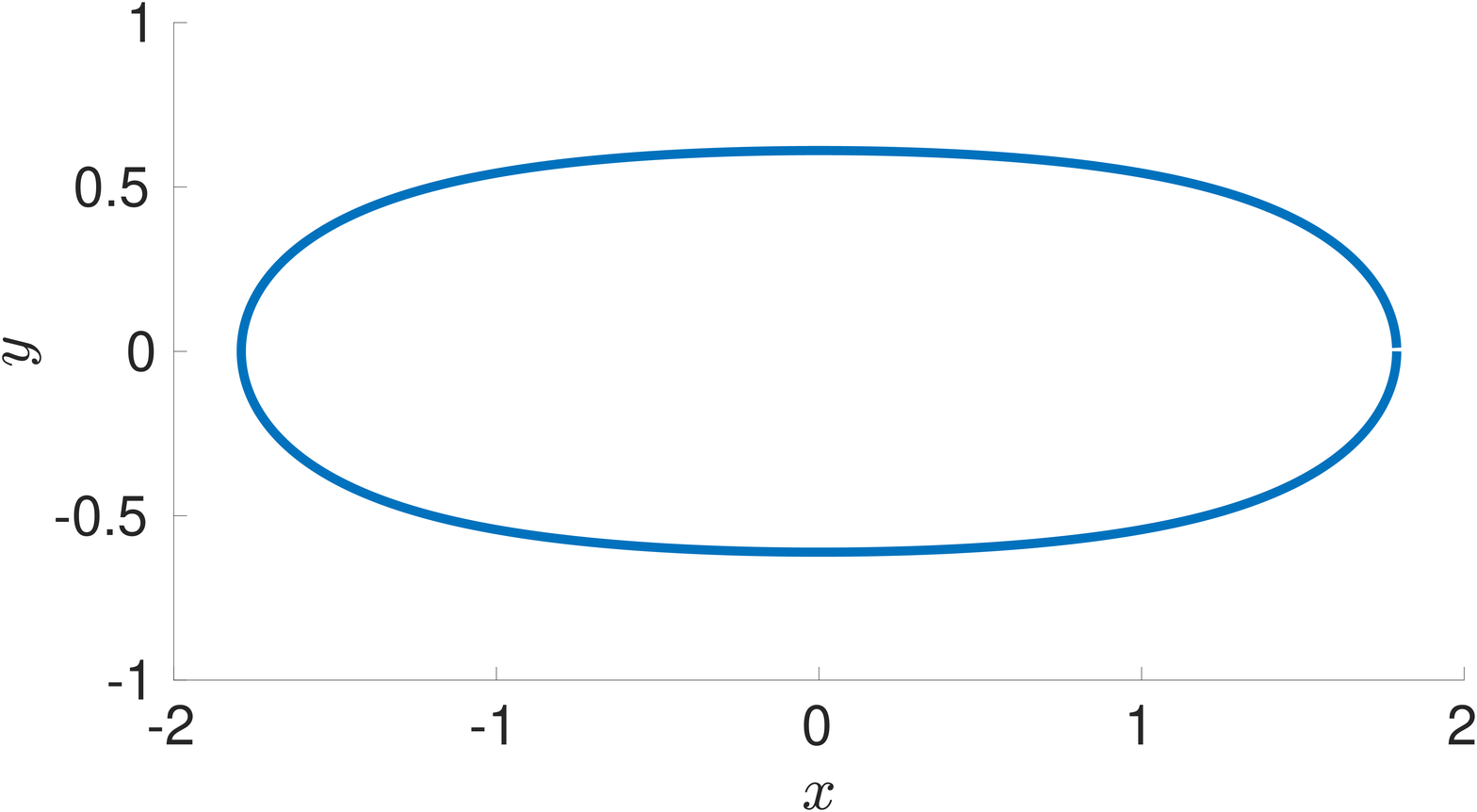}&  
 \includegraphics[width=0.25\textwidth]{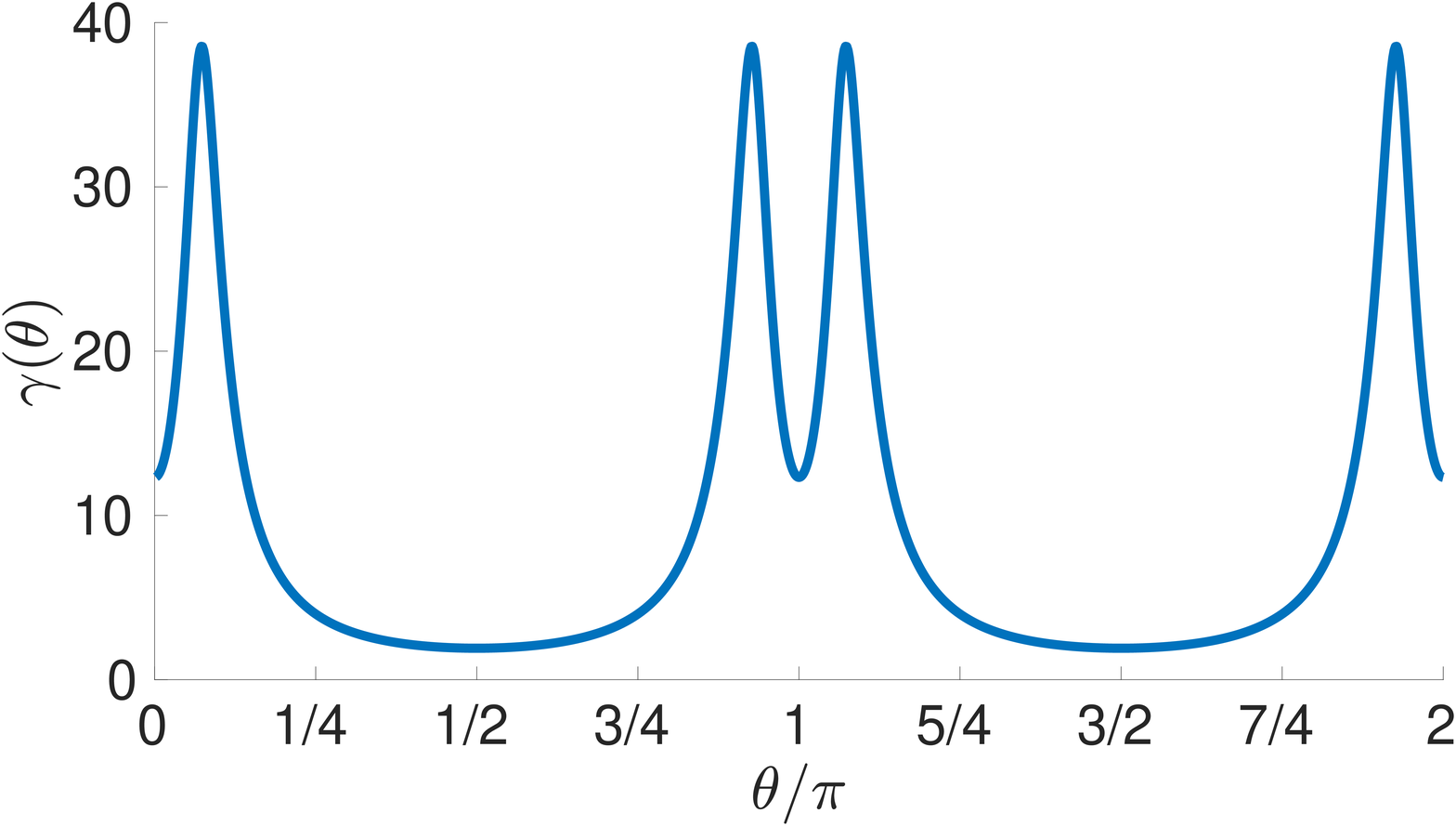}
 \includegraphics[width=0.25\textwidth]{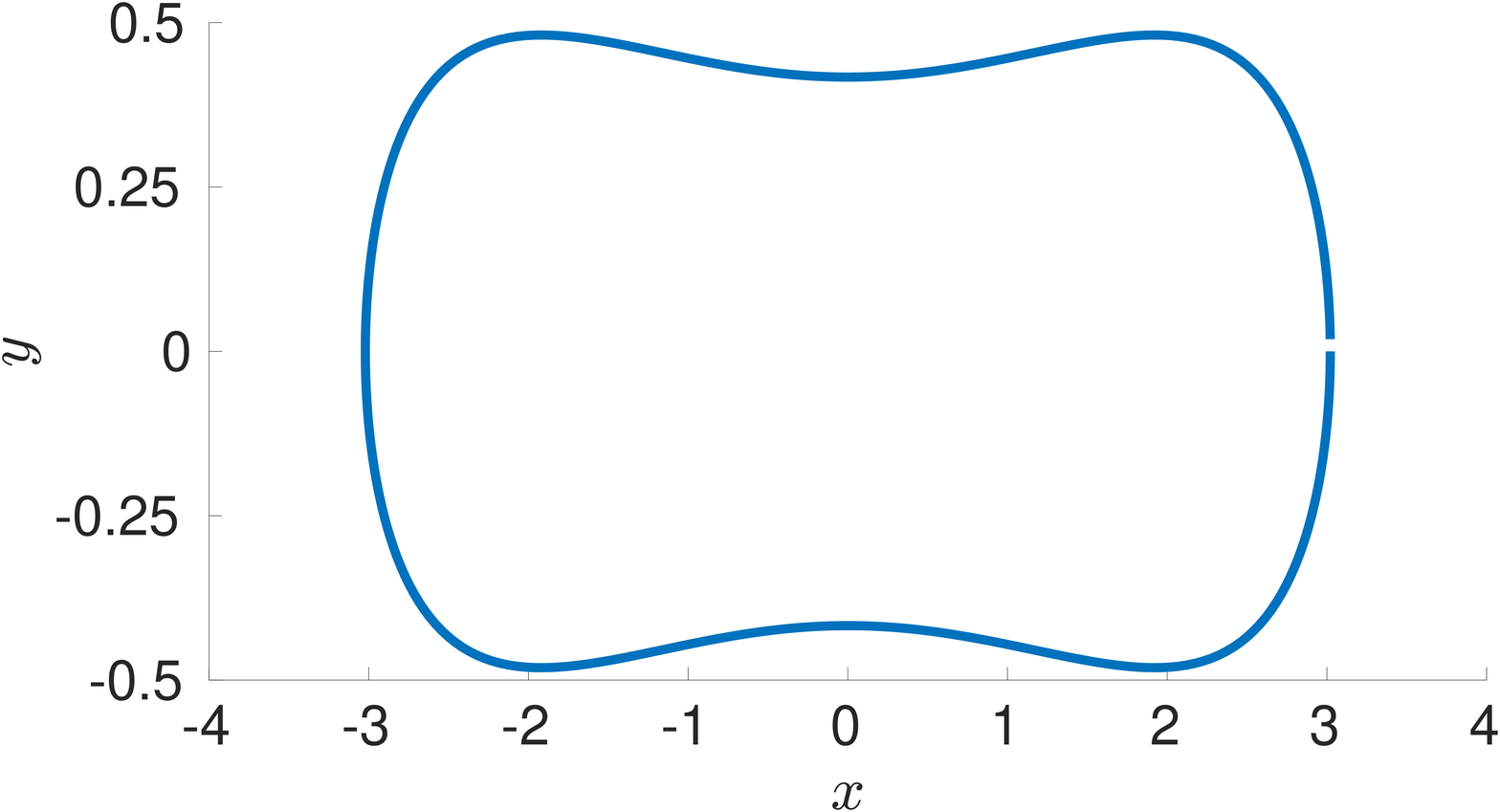}\\
 \qquad (c) $\gamma(\theta)$ and $z(\theta)$ for $r_1 = 0.525, b \sim 4.0954$ & \qquad (d) $\gamma(\theta)$ and $z(\theta)$ for $r_1 = 0.925, b \sim 9.3439$ \hspace{2.5cm}
\end{tabular}
 \caption{Panels (a)-(d): $\gamma(\theta)$ and $z(\theta)$ at the points $A$--$D$ highlighted in Figure \ref{fig_bifurcation}. In panel (b), $\gamma$ appears to tend to be concentrated only on the horizontal parts of $z$, leading to a possible solution consisting only of two symmetric curves (cf. \cite[Figure 1]{ONeil:relative-equilibria-vortex-sheets}) and a change of topology.
} \label{fig_panels}
\end{figure}

\appendix
\section{Derivatives of the Functional}\label{appendix}

\subsection{Functional derivatives}
Recall that $\mathcal{F}(b,g,r) = (\mathcal{F}_1,\mathcal{F}_2) $ is given in \eqref{mainfunctional}. For simplicity, we denote
 \begin{align*}
& A_1 := (b+g(\eta)),\\
& A_2 := \left(r'(\theta)\cos(\theta-\eta)-(1+r(\theta))\sin(\theta-\eta)\right)(1+r(\eta))-(1+r(\theta))r'(\theta),\\
& A_3 := \frac{1}{{(1+r(\theta))^2+(1+r(\eta))^2-2(1+r(\theta))(1+r(\eta))\cos(\theta-\eta)}},\\
& A_4:= r'(\theta)(1+r(\theta)),\\
& A_5 :=(b+g(\eta))(b+g(\theta)),\\
& A_6:= (1+r(\theta))^2-(r'(\theta)\sin(\theta-\eta)+(1+r(\theta))\cos(\theta-\eta))(1+r(\eta)),\\
& A_7 := \frac{1}{r'(\theta)^2+(1+r(\theta))^2},\\
& A_8 := \frac{(1+r(\theta))^2(b+g(\theta))}{r'(\theta)^2+(1+r(\theta))^2}.
 \end{align*}
 We also denote the average integral by $\fint f(\theta) d\theta := \frac{1}{2\pi} \int_{-\pi}^{\pi} f(\theta)d\theta$. Therefore the functional $\mathcal{F}$ can be written as 
  \begin{align}
 &\mathcal{F}_1 =  \fint A_1A_2A_3  d\eta + \Omega A_4 , \label{functionalf1} \\
 &\tilde{\mathcal{F}}_2 =\fint  A_5 A_6 A_7A_3  d\eta - \Omega A_8. \label{functionalf2}
 \end{align}
 
 We will expand $A_i(g,\theta,\eta)$ and $A_i(r,\theta,\eta)$ up to quadratic/cubic order in $g$ and $r$.
  
 \begin{lemma}\label{expansion}
 Let $A_i's$ be as above. We have
 \begin{align*}
 &A_1 = b+g(\eta), \\
 &A_2 = -\sin(\theta-\eta) + \left[ r'(\theta)(\cos(\theta-\eta)-1) - (r(\theta) + r(\eta))\sin(\theta-\eta) \right] \\
 & \quad + \left[ r'(\theta)r(\eta)(\cos(\theta-\eta)-1) - r'(\theta)(r(\theta)-r(\eta)) -r(\theta)r(\eta)\sin(\theta-\eta)\right], \\
 & A_3 = \frac{1}{2-2\cos(\theta-\eta)} - \frac{r(\theta)+r(\eta)}{2-2\cos(\theta-\eta)} \\
 & \quad + \frac{1}{2-2\cos(\theta-\eta)} \left[ r(\theta)^2+r(\theta)r(\eta) + r(\eta)^2 - \frac{(r(\theta)-r(\eta))^2}{2-2\cos(\theta-\eta)} \right] \\
  & \quad + \frac{1}{2-2\cos(\theta-\eta)} \left[ \frac{(r(\theta)+r(\eta))(r(\theta)-r(\eta))^2}{1-\cos(\theta-\eta)} - \left[r(\theta)^3 + r(\theta)^2r(\eta) + r(\theta)r(\eta)^2 + r(\eta)^3 \right]\right] + O(r^4),\\
 & A_4 = r'(\theta) + r(\theta)r'(\theta), \\
 & A_5 = b^2 + b(g(\theta)+g(\eta)) + g(\theta)g(\eta), \\
 & A_6 = (1-\cos(\theta-\eta)) + \left[ (r(\theta)+r(\eta))(1-\cos(\theta-\eta)) + (r(\theta)-r(\eta)) -r'(\theta)\sin(\theta-\eta)\right] \\
  & \quad + \left[ r(\theta)(r(\theta)-r(\eta)) + r(\theta)r(\eta)(1-\cos(\theta-\eta)) -r'(\theta)r(\eta)\sin(\theta-\eta) \right], \\
  &A_7 = 1 - 2r(\theta) + \left[ 3r(\theta)^2 - r'(\theta)^2 \right] + \left[ 4r(\theta)r'(\theta)^2 - 4r(\theta)^3 \right] + O(r^4), \\
  & A_8 = b+g(\theta) - br'(\theta)^2 + \left[ 2br(\theta)r'(\theta)^2 - g(\theta)r'(\theta)^2 \right]+ O(r^4+g^4).
 \end{align*}
 \end{lemma}
 
\begin{proof}
Straightforward.
\end{proof}

\subsubsection{Linear parts}

We denote by $A_i^j$ the $j$th order term in $A_i$. For example, $A_3^1 =  -\frac{r(\theta) + r(\eta)}{2-2\cos(\theta-\eta)}$.

\begin{lemma}\label{derivative}
Let $\mathcal{F}_i$'s and $A_{i}$'s be defined as before. Then
\begin{align}
&\frac{d}{dt} \mathcal{F}_1(b,tg,tr)\bigg|_{t=0} = -\fint \frac{g(\eta)\sin(\theta-\eta)}{2-2\cos(\theta-\eta)}  d\eta+ \left( \Omega-\frac{b}{2} \right) r'(\theta),\label{linearf1}\\
&\frac{d}{dt} \tilde{\mathcal{F}}_2(b,tg,tr)\bigg|_{t=0} = \left( \frac{b}{2}-\Omega \right) g(\theta) + b^2 \left( \fint \frac{r(\theta)-r(\eta)}{2-2\cos(\theta-\eta)} d\eta - r(\theta) \right). \label{linearf2}
\end{align}
\end{lemma}
\begin{proof}
We compute $\mathcal{F}_1$ first. In  view of \eqref{functionalf1}, we collect the linear terms in $A_1A_2A_3 + \Omega A_4$ from Lemma~\ref{expansion}.  Hence we have
\begin{align*}
\mathcal{F}_1(b,tg,tr) & = \fint -\frac{tg(\eta)\sin(\theta-\eta)}{2-2\cos(\theta-\eta)}  + bt\frac{r'(\theta)(\cos(\theta-\eta)-1) - (r(\theta) + r(\eta))\sin(\theta-\eta)}{2-2\cos(\theta-\eta)} \\
& \ +\frac{bt\sin(\theta-\eta)(r(\theta)+r(\eta))}{2-2\cos(\theta-\eta)}  d\eta + t\Omega r'(\theta)  + O(t^2) \\
 & = -\fint \frac{tg(\eta)\sin(\theta-\eta)}{2-2\cos(\theta-\eta)} d\eta + t \left( \Omega-\frac{b}{2} \right)  r'(\theta)  + O(t^2).
\end{align*}
 Thus we obtain \eqref{linearf1} by differentiating with respect to $t$.

In order to compute the derivative of $\tilde{\mathcal{F}}_2$, we collect the linear terms in \eqref{functionalf2} from Lemma~\ref{expansion} and obtain
\begin{align*}
\tilde{\mathcal{F}}_2(b,tg,tr) & = \fint   \frac{bt(g(\theta)+g(\eta))}{2} \\
& \quad +  tb^2 \left[ \frac{(r(\theta)+r(\eta))}{2} + \frac{(r(\theta)-r(\eta))}{2-2\cos(\theta-\eta)} -\frac{r'(\theta)\sin(\theta-\eta)}{2-2\cos(\theta-\eta)}\right] \\
& \quad - tb^2 r(\theta) -\frac{tb^2}{2}(r(\theta) + r(\eta))  - t\Omega g(\theta) d\eta + O(t^2)\\
& = t\left( \frac{b}{2} - \Omega \right)g(\theta) + tb^2\left( \fint \frac{r(\theta)-r(\eta)}{2-2\cos(\theta-\eta)}d\eta - r(\theta)\right) + O(t^2),
\end{align*}
where we used $\fint g(\eta)d\eta = \fint r(\eta) d\eta = 0$ and $\fint \frac{\sin(\theta-\eta)}{2-2\cos(\theta-\eta)} d\eta = 0$. By differentiating in $t$, we obtain the desired result \eqref{linearf2}. 
\end{proof}
  
  \subsubsection{Quadratic parts}
 Now we compute the quadratic expansion of $\mathcal{F}_1$ and $\tilde{\mathcal{F}}_2$.
 
 \begin{lemma}\label{quadratic_1}
 Let $\mathcal{F}_i$'s and $A_{i}$'s be defined as before. Then
 \begin{align}
 & \mathcal{F}_1 =   \left( \frac{b}{2} + \Omega \right) r(\theta)r'(\theta) -b\fint  \frac{r'(\theta)(r(\theta)-r(\eta))}{2-2\cos(\theta-\eta)} d\eta + b\fint \frac{(r(\theta)-r(\eta))^2\sin(\theta-\eta)}{(2-2\cos(\theta-\eta))^2} d\eta \nonumber \\
 & \quad \quad +\text{linear terms}+ O(r^3 + g^3), \label{quadraticf1}\\
 & \tilde{\mathcal{F}}_2 = \frac{3b^2}{2}r(\theta)^2  - bg(\theta)r(\theta) + b\left(\Omega-\frac{b}{2} \right) r'(\theta)^2  \nonumber\\
  & \quad \quad -\frac{b^2}{2}\fint \frac{(r(\theta)-r(\eta))(5r(\theta) + r(\eta))}{2-2\cos(\theta-\eta)} d\eta + b\fint \frac{(g(\theta)+g(\eta))(r(\theta) - r(\eta))}{2-2\cos(\theta-\eta)} d\eta \nonumber\\
   & \quad \quad - b \fint \frac{ g(\eta)r'(\theta) \sin(\theta-\eta)}{2-2\cos(\theta-\eta)} d\eta + \text{linear terms}+ O(r^3 + g^3). \label{quadraticf2}
 \end{align}
 \end{lemma}
 
 \begin{proof}
 We compute $\mathcal{F}_1$ first. By collecting quadratic terms in \eqref{functionalf1} from Lemma~\ref{expansion}, we have (we will have, for example, $A_1^2A_2^0A_3^0 + A_1^0A_2^2A_3^0 + A_1^0A_2^0A_3^2 +  A_1^1A_2^1 A_3^0 + A_1^1A_2^0A_3^1 + A_1^0A_2^1A_3^1 + \Omega A_4^2$).
 \begin{align*}
 \mathcal{F}_1 & =\fint b\left[-\frac{1}{2}r'(\theta)r(\eta) - \frac{r'(\theta)(r(\theta)-r(\eta))}{2-2\cos(\theta-\eta)} - \frac{r(\theta)r(\eta)\sin(\theta-\eta)}{2-2\cos(\theta-\eta)} \right] \\
 & \quad -\frac{b\sin(\theta-\eta)}{2-2\cos(\theta-\eta)}\left[  r(\theta)^2+r(\theta)r(\eta) + r(\eta)^2 - \frac{(r(\theta)-r(\eta))^2}{2-2\cos(\theta-\eta)}\right] \\
 & \quad +  \left[ - \frac{g(\eta)r'(\theta)}{2} - \frac{g(\eta)(r(\theta)+r(\eta))\sin(\theta-\eta)}{2-2\cos(\theta-\eta)} \right] + \frac{g(\eta)(r(\theta) + r(\eta))\sin(\theta-\eta)}{2-2\cos(\theta-\eta)} \\
 & \quad + \left[ \frac{br'(\theta)(r(\theta)+r(\eta))}{2} + \frac{b(r(\theta)+r(\eta))^2\sin(\theta-\eta)}{2-2\cos(\theta-\eta)}\right] + \Omega r(\theta)r'(\theta) d\eta + O(r^3 + g^3) \\
  & = \fint \left( \frac{b}{2} + \Omega \right) r(\theta)r'(\theta) - \frac{br'(\theta)(r(\theta)-r(\eta))}{2-2\cos(\theta-\eta)} + \frac{b(r(\theta)-r(\eta))^2\sin(\theta-\eta)}{(2-2\cos(\theta-\eta))^2} d\eta \\
  & \quad + \text{linear terms}+ O(r^3 + g^3),
 \end{align*}
which yields \eqref{quadraticf1}. Now we will expand $\tilde{\mathcal{F}}_2$ up to the quadratic order. By collecting all quadratic terms in \eqref{functionalf2} from Lemma~\ref{expansion}, we obtain (we will have $A_5^2A_6^0A_7^0A_3^0  + A_5^0A_6^2A_7^0A_3^0 + A_5^0A_6^0A_7^2A_3^0 + A_5^0A_6^0A_7^0A_3^2 +  A_5^1A_6^1A_7^0A_3^0 + A_5^1A_6^0A_7^1A_3^0 + A_5^1A_6^0A_7^0A_3^1 + A_5^0A_6^1A_7^1A_3^0 + A_5^0A_6^1A_7^0A_3^1  + A_5^0A_6^0A_7^1A_3^1 - \Omega A_8^2$),
\begin{align*}
\tilde{\mathcal{F}}_2 & = \fint \frac{g(\theta)g(\eta)}{2} +  \left[ b^2\frac{r(\theta)(r(\theta)-r(\eta))}{2-2\cos(\theta-\eta)} + b^2\frac{r(\theta)r(\eta)}{2} - b^2\frac{r'(\theta)r(\eta)\sin(\theta-\eta)}{2-2\cos(\theta-\eta)} \right]\\
  & \quad + b^2\frac{(3r(\theta)^2 - r'(\theta)^2)}{2}   +  \left[ \frac{b^2(r(\theta)^2 + r(\theta)r(\eta) + r(\eta)^2)}{2} - \frac{b^2}{2}\frac{(r(\theta)-r(\eta))^2}{2-2\cos(\theta-\eta)} \right] \\
   & \quad + \left[ b\frac{(g(\theta)+g(\eta))(r(\theta)+r(\eta))}{2} + b\frac{(g(\theta)+g(\eta))(r(\theta)-r(\eta))}{2-2\cos(\theta-\eta)} - b\frac{(g(\theta)+g(\eta))r'(\theta)\sin(\theta-\eta)}{2-2\cos(\theta-\eta)} \right] \\
   & \quad   -b((g(\theta)+g(\eta)))r(\theta) - b\frac{(g(\theta)+g(\eta))(r(\theta)+r(\eta))}{2} \\
   & \quad + \left[ -b^2 r(\theta)(r(\theta) + r(\eta)) -2b^2\frac{r(\theta)(r(\theta)-r(\eta))}{2-2\cos(\theta-\eta)} + 2b^2 \frac{r(\theta)r'(\theta)\sin(\theta-\eta)}{2-2\cos(\theta-\eta)} \right] \\
   & \quad + \left[ -b^2\frac{(r(\theta)+r(\eta))^2}{2} - b^2 \frac{(r(\theta)^2 - r(\eta)^2)}{2-2\cos(\theta-\eta)} + b^2 \frac{r'(\theta)(r(\theta)+r(\eta))\sin(\theta-\eta)}{2-2\cos(\theta-\eta)}\right] \\
   & \quad + b^2r(\theta)(r(\theta) + r(\eta)) + b\Omega r'(\theta)^2 d\eta + O (r^3 + g^3)\\
   &  = \fint  b^2\frac{(3r(\theta)^2 - r'(\theta)^2)}{2}  - bg(\theta)r(\theta) + b\Omega r'(\theta)^2  \\
   & \quad  + \frac{1}{2-2\cos(\theta-\eta)} \left(-\frac{b^2}{2}(r(\theta)-r(\eta))(5r(\theta) + r(\eta)) + b(g(\theta)+g(\eta))(r(\theta) - r(\eta)) \right)\\
   & \quad - \frac{ bg(\eta)r'(\theta) \sin(\theta-\eta)}{2-2\cos(\theta-\eta)} d\eta + \text{linear terms}+ O(r^3 + g^3),
\end{align*} 
where we used $\fint g(\eta) d\eta =\fint r(\eta) d\eta= \fint \frac{\sin(\theta-\eta)}{2-2\cos(\theta-\eta)} d\eta = 0$. This yields the desired result \eqref{quadraticf2}.
\end{proof}

\begin{lemma}\label{second_derivative}
 Let $\mathcal{F}_i$'s and $A_{i}$'s be defined as before. Then
 \begin{align}
 \frac{d^2}{dtds} \mathcal{F}_1(b,tg+s\tilde{g}&,tr+s\tilde{r}) \bigg|_{t=s=0}  =\left(\frac{b}{2} +\Omega \right)\left( r(\theta)\tilde{r}'(\theta) + \tilde{r}(\theta)r'(\theta) \right)  \nonumber \\ 
 & \quad   -b\fint \frac{r'(\theta)(\tilde{r}(\theta) - \tilde{r}(\eta))}{2-2\cos(\theta-\eta)} d\eta -b\fint \frac{\tilde{r}'(\theta)(r(\theta) - r(\eta))}{2-2\cos(\theta-\eta)} d\eta \nonumber \\
 &   \quad+ 2b\fint \frac{(r(\theta)-r(\eta))(\tilde{r}(\theta)-\tilde{r}(\eta))\sin(\theta-\eta)}{(2-2\cos(\theta-\eta))^2}, \label{second_derivative_f1} 
\end{align}
and
\begin{align}
 \frac{d^2}{dtds} \tilde{\mathcal{F}}_2(b,tg+s\tilde{g}&,tr+s\tilde{r}) \bigg|_{t=s=0}  =  3b^2r(\theta)\tilde{r}(\theta) - b(g(\theta)\tilde{r}(\theta) + \tilde{g}(\theta)r(\theta)) + 2b \left(\Omega - \frac{b}{2}\right) r'(\theta)\tilde{r}'(\theta) \nonumber \\
 & \quad   -\frac{b^2}{2}\fint \frac{(r(\theta)-r(\eta))(5\tilde{r}(\theta)+\tilde{r}(\eta))}{2-2\cos(\theta-\eta)}d\eta -\frac{b^2}{2}\fint \frac{(\tilde{r}(\theta)-\tilde{r}(\eta))(5{r}(\theta)+{r}(\eta))}{2-2\cos(\theta-\eta)}d\eta \nonumber \\
 & \quad + b\fint \frac{(g(\theta)+g(\eta))(\tilde{r}(\theta)-\tilde{r}(\eta))}{2-2\cos(\theta-\eta)} d\eta +  b\fint \frac{(\tilde{g}(\theta)+\tilde{g}(\eta))({r}(\theta)-{r}(\eta))}{2-2\cos(\theta-\eta)} d\eta \nonumber \\
 & \quad - b \fint \frac{g(\eta)\tilde{r}'(\theta)\sin(\theta-\eta)}{2-2\cos(\theta-\eta)} d\eta - b \fint \frac{\tilde{g}(\eta){r}'(\theta)\sin(\theta-\eta)}{2-2\cos(\theta-\eta)} d\eta. \label{second_derivative_f2}
 \end{align}
\end{lemma}

\begin{proof}
We compute $\mathcal{F}_1$ first. From \eqref{quadraticf1} in Lemma~\ref{quadratic_1}, we collect all $st$ terms and obtain
\begin{align*}
\mathcal{F}_1(b,tg+s\tilde{g},tr+s\tilde{r}) & =  st \left[ \left(\frac{b}{2} +\Omega \right)\left( r(\theta)\tilde{r}'(\theta) + \tilde{r}(\theta)r'(\theta) \right) \right . \\ 
 & \left. \quad -b\fint \frac{r'(\theta)(\tilde{r}(\theta) - \tilde{r}(\eta))}{2-2\cos(\theta-\eta)} d\eta -b\fint \frac{\tilde{r}'(\theta)(r(\theta) - r(\eta))}{2-2\cos(\theta-\eta)} d\eta \right. \\
 & \left. \quad + 2b\fint \frac{(r(\theta)-r(\eta))(\tilde{r}(\theta)-\tilde{r}(\eta))}{(2-2\cos(\theta-\eta))^2}  \right] + \text{linear terms} + O(t^2 + s^2).
\end{align*}
Once we differentiate the above equation with respect to  $t$ and $s$,  the desired result \eqref{second_derivative_f1} follows immediately. Similarly, we  collect all st terms from \eqref{quadraticf2} and obtain
\begin{align*}
 \tilde{\mathcal{F}}_2(b,tg+s\tilde{g},tr+s\tilde{r}) & = st \left[ 3b^2r(\theta)\tilde{r}(\theta) - b(g(\theta)\tilde{r}(\theta) + \tilde{g}(\theta)r(\theta)) + 2b \left(\Omega - \frac{b}{2}\right) r'(\theta)\tilde{r}'(\theta) \right. \\
 & \left. \quad  -\frac{b^2}{2}\fint \frac{(r(\theta)-r(\eta))(5\tilde{r}(\theta)+\tilde{r}(\eta))}{2-2\cos(\theta-\eta)}d\eta -\frac{b^2}{2}\fint \frac{(\tilde{r}(\theta)-\tilde{r}(\eta))(5{r}(\theta)+{r}(\eta))}{2-2\cos(\theta-\eta)}d\eta \right. \\
 & \left. \quad + b\fint \frac{(g(\theta)+g(\eta))(\tilde{r}(\theta)-\tilde{r}(\eta))}{2-2\cos(\theta-\eta)} d\eta +  b\fint \frac{(\tilde{g}(\theta)+\tilde{g}(\eta))({r}(\theta)-{r}(\eta))}{2-2\cos(\theta-\eta)} d\eta \right. \\
 & \left. \quad - b \fint \frac{g(\eta)\tilde{r}'(\theta)\sin(\theta-\eta)}{2-2\cos(\theta-\eta)} d\eta - b \fint \frac{\tilde{g}(\eta){r}'(\theta)\sin(\theta-\eta)}{2-2\cos(\theta-\eta)} d\eta \right] \\
 &\quad + \text{linear terms} + O(t^2 + s^2). 
\end{align*}
Once we differentiate the above equation with respect to $t$ and $s$,  the desired result \eqref{second_derivative_f2} follows immediately.
\end{proof}

 \subsubsection{Cubic parts}
 We will expand $\tilde{\mathcal{F}}_2$ up to  cubic order with respect to the $r$ variable (we fix $g = 0$). We denote $B:=A_3A_6$ so that \eqref{functionalf2} can be written as (with $g = 0$)
 \begin{align}\label{functionalf21}
 \tilde{\mathcal{F}}_2 = b^2\fint B d\eta A_7 - \Omega A_8. 
 \end{align}
 We will first expand $B$ up to  cubic order. 
  \begin{lemma}\label{cubic1}
 Let $\tilde{\mathcal{F}}_2$, $A_i$ and $B$ be as defined as before. Then
 \begin{align*}
 \fint B d\eta & = \frac{1}{2} + \fint \frac{r(\theta)-r(\eta)}{2-2\cos(\theta-\eta)} d\eta - \frac{1}{2}\fint \frac{r(\theta)^2 - r(\eta)^2}{2-2\cos(\theta-\eta)} d\eta \\
  &  + \left[ \frac{1}{2} \fint \frac{(r(\theta)-r(\eta))(r(\theta)^2 + r(\eta)^2)}{2-2\cos(\theta-\eta)}d\eta - \fint \frac{(r(\theta)-r(\eta))^3}{(2-2\cos(\theta-\eta))^2} d\eta + \fint \frac{r'(\theta)(r(\theta)-r(\eta))^2\sin(\theta-\eta)}{(2-2\cos(\theta-\eta))^2} d\eta \right]\\
  &  + O(r^4).
 \end{align*}
 \end{lemma}
 
 \begin{proof}
Using \eqref{expansion}, we will compute the constant ($=: B^0$), linear ($=: B^1$), quadratic ($=: B^2$) and cubic ($=: B^3$) terms of $B = A_3 A_6$ separately. It is straightforward that
\begin{align}\label{0thorder}
\fint B^0 d\eta = \frac{1}{2}.
\end{align}
For $B^1$, we compute $B^1 (= A_3^1 A_6^0 + A_3^0 A_6^1)$,
\begin{align}\label{1storder}
\fint B^1 d\eta = \fint \frac{r(\theta)-r(\eta)}{2-2\cos(\theta-\eta)} - \frac{r'(\theta)\sin(\theta-\eta)}{2-2\cos(\theta-\eta)}d\eta = \fint \frac{r(\theta)-r(\eta)}{2-2\cos(\theta-\eta)} d\eta,
\end{align}
where we used $\fint \frac{\sin(\theta-\eta)}{2-2\cos(\theta-\eta)} d\eta = 0$. For $B^2$, we compute $A_3^2A_6^0 + A_3^0A_6^2 + A_3^1A_6^1$, hence
\begin{align}\label{secondorder}
\fint B^2 d\eta & = \left[ \frac{r(\theta)^2 + r(\theta)r(\eta) + r(\eta)^2}{2} - \frac{1}{2}\fint \frac{(r(\theta)-r(\eta))^2}{2-2\cos(\theta-\eta)} d\eta \right] \nonumber\\
& \quad + \left[ \fint \frac{r(\theta)(r(\theta)-r(\eta))}{2-2\cos(\theta-\eta)} + \frac{r(\theta)r(\eta)}{2} - \frac{r'(\theta)r(\eta)\sin(\theta-\eta)}{2-2\cos(\theta-\eta)}d\eta \right] \nonumber\\
& \quad + \left[ \fint -\frac{1}{2}(r(\theta)+r(\eta))^2 - \frac{r(\theta)^2 - r(\eta)^2}{2-2\cos(\theta-\eta)} + \frac{r'(\theta)(r(\theta)+r(\eta))\sin(\theta-\eta)}{2-2\cos(\theta-\eta)} d\eta  \right] \nonumber\\
& = -\frac{1}{2}\fint \frac{(r(\theta)^2-r(\eta)^2)}{2-2\cos(\theta-\eta)} d\eta,
\end{align}
where we used $\fint \frac{\sin(\theta-\eta)}{2-2\cos(\theta-\eta)}d\eta = 0$. For $B^3$, we compute $B^3 (= A_3^3A_6^0 + A_3^2 A_6^1 + A_3^1A_6^2 + A_3^0A_6^3)$,
\begin{align}\label{thirdorder}
\fint B^3 d\eta  & = \left[ \fint \frac{(r(\theta)+r(\eta))(r(\theta)-r(\eta))^2}{2-2\cos(\theta-\eta)} d\eta - \frac{1}{2} \fint r(\theta)^3 + r(\theta)^2r(\eta) + r(\theta)r(\eta)^2 + r(\eta)^3  d\eta  \right] \nonumber\\
 & \quad + \left[ \fint \frac{(r(\theta)+r(\eta))(r(\theta)^2 + r(\theta)r(\eta) + r(\eta)^2)}{2} d\eta - \frac{1}{2}\fint \frac{(r(\theta)+r(\eta))(r(\theta)-r(\eta))^2}{2-2\cos(\theta-\eta)} d\eta \right. \nonumber\\
 & \left. \quad + \fint \frac{r(\theta)^3 -r(\eta)^3}{2-2\cos(\theta-\eta)}d\eta - \fint \frac{(r(\theta)-r(\eta))^3}{(2-2\cos(\theta-\eta))^2} d\eta  \right. \nonumber \\
 & \left. \quad - \fint \frac{r'(\theta)(r(\theta)^2 + r(\theta)r(\eta) + r(\eta)^2)\sin(\theta-\eta)}{2-2\cos(\theta-\eta)} d\eta + \fint \frac{r'(\theta)(r(\theta)-r(\eta))^2\sin(\theta-\eta)}{(2-2\cos(\theta-\eta))^2} d\eta \right]  \nonumber \\
 & \quad + \left[  -\fint \frac{r(\theta)(r(\theta)^2-r(\eta)^2)}{2-2\cos(\theta-\eta)} d\eta - \fint \frac{r(\theta)r(\eta)(r(\theta)+r(\eta))}{2} d\eta + \fint \frac{r'(\theta)r(\eta)(r(\theta) + r(\eta))\sin(\theta-\eta)}{2-2\cos(\theta-\eta)} d\eta \right] \nonumber \\
  & = \frac{1}{2}\fint \frac{(r(\theta)-r(\eta))(r(\theta)^2 + r(\eta)^2)}{2-2\cos(\theta-\eta)} d\eta - \fint \frac{(r(\theta)-r(\eta))^3}{(2-2\cos(\theta-\eta))^2} d\eta + \fint \frac{r'(\theta)(r(\theta)-r(\eta))^2\sin(\theta-\eta)}{(2-2\cos(\theta-\eta))^2}.
\end{align}
Thus the desired result follows from \eqref{0thorder}, \eqref{1storder}, \eqref{secondorder} and \eqref{thirdorder}. 
\end{proof}

\begin{lemma}\label{cubic_1}
Let $\tilde{\mathcal{F}}_2$, $A_i$'s and $B$ be as defined as before. Then 
\begin{align*}
\frac{1}{6}\frac{d^3}{dt^3}\tilde{\mathcal{F}}_2(b,0,tr)\bigg|_{t=0} &  = \frac{b^2}{2} \fint \frac{(r(\theta) - r(\eta))(3r(\theta)^2 + 2 r(\theta) r(\eta) + r(\eta)^2)}{2-2\cos(\theta-\eta)} d\eta - b^2 \fint \frac{(r(\theta)-r(\eta))^3}{(2-2\cos(\theta-\eta))^2} d\eta \\
& \quad +  b^2 \fint \frac{r'(\theta)(r(\theta)-r(\eta))^2\sin(\theta-\eta)}{(2-2\cos(\theta-\eta))^2} \eta + b^2\fint \frac{(3r(\theta)^2 - r'(\theta)^2)(r(\theta)-r(\eta))}{2-2\cos(\theta-\eta)} d\eta \\
& \quad + 2b (b-\Omega) r(\theta)r'(\theta)^2 - 2b^2r(\theta)^3 
\end{align*}
\end{lemma}

\begin{proof}
We first collect all cubic terms of $\tilde{\mathcal{F}}_2 (b,0,r)$ in $r$. From \eqref{functionalf21}, we have that the cubic terms consist of $b^2 (B^3A_7^0 + B^2A_7^1 + B^1A_7^2 + B^0 A_7^3 )- \Omega A_8^3$. Using Lemma~\ref{expansion} and~\ref{cubic1} and the fact that $A_7$ does not depend on $\eta$, we obtain
\begin{align*}\hspace*{-0.8cm}
\tilde{\mathcal{F}}_2 (b,0,r) & = b^2\left[ \frac{1}{2} \fint \frac{(r(\theta)-r(\eta))(r(\theta)^2 + r(\eta)^2)}{2-2\cos(\theta-\eta)}d\eta - \fint \frac{(r(\theta)-r(\eta))^3}{(2-2\cos(\theta-\eta))^2} d\eta + \fint \frac{r'(\theta)(r(\theta)-r(\eta))^2\sin(\theta-\eta)}{(2-2\cos(\theta-\eta))^2} \eta \right] \\
& \quad + b^2\fint \frac{r(\theta)(r(\theta)^2-r(\eta)^2)}{2-2\cos(\theta-\eta)} d\eta + b^2\fint \frac{(3r(\theta)^2 - r'(\theta)^2)(r(\theta)-r(\eta))}{2-2\cos(\theta-\eta)} d\eta + b^2 \left[2r(\theta)r'(\theta)^2 -2r(\theta)^3 \right]\\
& \quad -2b\Omega r(\theta)r'(\theta)^2 + \text{lower order terms} + O(r^4)\\
& = \frac{b^2}{2} \fint \frac{(r(\theta) - r(\eta))(3r(\theta)^2 + 2 r(\theta) r(\eta) + r(\eta)^2)}{2-2\cos(\theta-\eta)} d\eta - b^2 \fint \frac{(r(\theta)-r(\eta))^3}{(2-2\cos(\theta-\eta))^2} d\eta \\
& \quad +  b^2 \fint \frac{r'(\theta)(r(\theta)-r(\eta))^2\sin(\theta-\eta)}{(2-2\cos(\theta-\eta))^2} \eta + b^2\fint \frac{(3r(\theta)^2 - r'(\theta)^2)(r(\theta)-r(\eta))}{2-2\cos(\theta-\eta)}  d\eta\\
& \quad + 2b (b-\Omega) r(\theta)r'(\theta)^2 - 2b^2r(\theta)^3 + \text{lower order terms} + O(r^4).
\end{align*}
Therefore, the desired result follows immediately.
\end{proof}

\subsection{Derivatives of the reduced functional}
We denote $v:= (0,\cos(2\theta))$. Given a pair of functions $(g,r)$, we denote $Q$ be the projection to the second mode of $r$, that is, $Q(g,r) := \left(\frac{1}{\pi} \int r(\theta)\cos(2\theta) d\theta \right) \cos(2\theta)$.
\begin{lemma}\label{somevalues}
Let $\mathcal{F}$, $v$, $Q$ be defined as before. We fix $b=2$ and $\Omega = 1$.  Then, 
\begin{align}
&\partial_b D\mathcal{F}(2,0)v =  ( \sin(2\theta), 0 ), \label{value1}\\
&\frac{1}{2}\frac{d^2}{dt^2}\mathcal{F}(2,tv) \bigg|_{t=0} = (-2\sin(4\theta) , -3\cos(4\theta)), \label{value2} \\
&\partial_b Q\frac{d^2}{dt^2} \mathcal{F}(b,tv) \bigg|_{b=2,t=0} = (0,0), \label{value10}\\
& \tilde{v} := - \left[ D \mathcal{F}(2,0) \right]^{-1} (I - Q)\partial_b D \mathcal{F}(2,0)v  = (2\cos(2\theta),0), \label{value7}\\
& \hat{v} := -\left[ D\mathcal{F}(2,0) \right]^{-1} \frac{d^2}{dt^2}\left[ (I - Q) \mathcal{F}(2,tv) \right] \bigg|_{t=0} =  \left( -8 \cos(4\theta), \frac{3}{2}\cos(4\theta) \right), \label{value4}\\
& \frac{d^2}{dtds} Q\mathcal{F}(2,tv+s\hat{v}) \bigg|_{t=s=0} = (0, -12\cos(2\theta)),  \label{value5}\\
& \frac{1}{3}\frac{d^3}{dt^3} Q \mathcal{F}(2,tv) \bigg|_{t=0} = \left( 0, 4\cos(2\theta) \right), \label{value6}\\
& Q\frac{d^2}{dtds}\mathcal{F}(2,tv+s\tilde{v}) = (0,0), \label{value8}\\
&\frac{1}{2}Q\partial_bD\mathcal{F}(2,0)\hat{v} = (0,0), \label{value9}\\
& 2Q\partial_bD\mathcal{F}(2,0)\tilde{v} = (0,2\cos(2\theta)). \label{value11}
\end{align}
\end{lemma}

\begin{proof}
To prove \eqref{value1}, it follows from Lemma~\ref{derivative} that
\begin{align*}
\partial_b D\mathcal{F}(b,0)v = \left( \sin(2\theta), 2b \left(\fint \frac{\cos(2\theta)-\cos(2\eta)}{2-2\cos(\theta-\eta)}d\eta - \cos(2\theta)\right)\right) = \left( \sin(2\theta) , 0\right),
\end{align*}
where the last equality follows from \eqref{lemma1}.

 To prove \eqref{value2}, note that $\frac{1}{2}\frac{d^2}{dt^2}\mathcal{F}(2,tv) \bigg|_{t=0} = \left( \frac{1}{2}\frac{d^2}{dt^2}\mathcal{F}_1(2,tv) \bigg|_{t=0} , \frac{1}{2}(I-P_0)\frac{d^2}{dt^2}\tilde{\mathcal{F}}_2(2,tv)\bigg|_{t=0} \right)$. Hence it follows from Lemma~\ref{quadratic_1} that
 \begin{align*}
 \frac{1}{2}\frac{d^2}{dt^2}\mathcal{F}_1(2,tv) \bigg|_{t=0} & = -4\cos(2\theta) \sin(2\theta)+ 4\sin(2\theta) \fint \frac{\cos(2\theta) - \cos(2\eta)}{2-2\cos(\theta-\eta)} d\eta\\
 & \quad  +2 \fint \frac{(\cos(2\theta)-\cos(2\eta))^2\sin(\theta-\eta)}{(2-2\cos(\theta-\eta))^2}d\eta \\
 & = -4\cos(2\theta)\sin(2\theta) + 4\sin(2\theta)\cos(2\theta) -2\sin(4\theta)\\
 & = -2\sin(4\theta),
 \end{align*}
 where the second equality follows from \eqref{lemma1} and Lemma~\ref{lemma5}.
 Also, Lemma~\ref{quadratic_1} gives 
 \begin{align*}
\frac{1}{2}(I-P_0)\frac{d^2}{dt^2}\tilde{\mathcal{F}}_2(2,tv)\bigg|_{t=0}  & = (I-P_0) \left[ 6\cos^2(2\theta)\right] - 2(I-P_0) \fint \frac{(\cos(2\theta) - \cos(2\eta))(5\cos(2\theta) + \cos(2\eta))}{2-2\cos(\theta-\eta)}d\eta \\
& = -4 (I-P_0) \cos^2(2\theta) - 2(I-P_0) \fint \frac{(\cos(2\theta) - \cos(2\eta))\cos(2\eta)}{2-2\cos(\theta-\eta)} d\eta \\
& = -3 \cos(4\theta),
 \end{align*}
 where the second equality follows from ~\eqref{lemma1} and the last equality follows from Lemma~\ref{lemma3}. Therefore we obtain \eqref{value2}.

 To prove \eqref{value10}, we can repeat the above computation and find that $\frac{d^2}{dt^2}\tilde{\mathcal{F}}_2(b,tv) \in \text{span}\left\{ \cos(4\theta)\right\}$, independently of $b$. By projecting it to the space of the second mode, we obtain \eqref{value10}. 
 
  To prove \eqref{value7}, note that $(I-Q)\partial_bD\mathcal{F}(2,0)v = (\sin(2\theta),0)$, which follows from \eqref{value1}. Also, it follows from Lemma~\ref{derivative} and \eqref{lemma2} that
  \begin{align*}
  D\mathcal{F}(2,0)(-2\cos(2\theta) , 0 ) = (\sin(2\theta),0) = (I-Q)\partial_bD\mathcal{F}(2,0)v.
  \end{align*}
  This immediately implies \eqref{value7}.
 
 To prove \eqref{value4}, we use Lemma~\ref{derivative} and \eqref{lemma1} and \eqref{lemma2} to obtain
 \begin{align*}
 D\mathcal{F}(2,0) \left(8\cos(4\theta), -\frac{3}{2}\cos(4\theta) \right) = (-4\sin(4\theta), -6\cos(4\theta)) = \frac{d^2}{dt^2}\mathcal{F}(2,tv)\bigg|_{t=0},
 \end{align*}
 where the last equality follows from \eqref{value2}. Therefore we obtain
 \begin{align*}
 - \left[ D\mathcal{F}(2,0) \right]^{-1} \frac{d^2}{dt^2}\left[ (I - Q) \mathcal{F}(2,tv) \right] \bigg|_{t=0} = \left( -8 \cos(4\theta), \frac{3}{2}\cos(4\theta) \right),
 \end{align*}
 which proves \eqref{value4}.
 
  To prove \eqref{value5},  note that $  \frac{d^2}{dtds}Q\mathcal{F} = \left( 0 , P_2\frac{d^2}{dtds}\tilde{\mathcal{F}}_2 \right)$. Therefore, it follows from \eqref{value4} and \eqref{second_derivative_f2} in Lemma~\ref{second_derivative} that (plugging $g=0,\ \tilde{g}=-8\cos(4\theta),\ r=\cos(2\theta)$, and $\tilde{r} = \frac{3}{2}\cos(4\theta)$)
  \begin{align*}
  \frac{d^2}{dtds}P_2\mathcal{F}_2(2,tv+s\tilde{v}) & =P_2 \left[ \left( 18\cos(2\theta)\cos(4\theta) +16\cos(2\theta)\cos(4\theta) \right) \right. \\
  &\left.  \quad -2 \fint \frac{(\cos(2\theta)-\cos(2\eta))(\frac{15}{2}\cos(4\theta) + \frac{3}{2}\cos(4\eta))}{2-2\cos(\theta-\eta)} d\eta \right. \\
  &\left. \quad  - 3 \fint \frac{(\cos(4\theta)-\cos(4\eta))(5\cos(2\theta)+\cos(2\eta))}{2-2\cos(\theta-\eta)} d\eta \right. \\
  & \left.  \quad -16 \fint \frac{(\cos(4\theta)+\cos(4\eta))(\cos(2\theta)-\cos(2\eta))}{2-2\cos(\theta-\eta)} d\eta \right. \\
  & \left. \quad -32 \fint \frac{\cos(4\eta)\sin(2\theta)\sin(\theta-\eta)}{2-2\cos(\theta-\eta)}d\eta \right] \\
  &=: P_2K_1 + P_2K_2 + P_2K_3 + P_2K_4 + P_2K_5.
  \end{align*}
 For $K_1$, we compute
 \begin{align}\label{k1}
 P_2K_1 = P_2 (34 \cos(2\theta)\cos(4\theta)) = 17 P_2 (\cos(2\theta) +\cos(6\theta)) = 17\cos(2\theta).
 \end{align}
 For $K_2$ we compute
 \begin{align}\label{k2}
 P_2K_2 & = P_2 \left( -15\cos(4\theta)\fint \frac{\cos(2\theta)-\cos(2\eta)}{2-2\cos(\theta - \eta)} d\eta -3\fint \frac{(\cos(2\theta)-\cos(2\eta))\cos(4\eta)}{2-2\cos(\theta-\eta)} d\eta \right) \nonumber \\
 & = P_2 \left( -15\cos(2\theta)\cos(4\theta) + \frac{3}{2}(\cos(2\theta) - \cos(6\theta))\right)  \nonumber \\
 & = P_2 (-6\cos(2\theta) - 9 \cos(6\theta)) \nonumber\\
 & = -6\cos(2\theta),
 \end{align}
 where the second equality follows from \eqref{lemma1} and Lemma~\ref{lemma3}.
 For $K_3$, we compute
 \begin{align}\label{k3}
 P_2 K_3 & = P_2\left( -15 \cos(2\theta) \fint \frac{\cos(4\theta)-\cos(4\eta)}{2-2\cos(\theta-\eta)} d\eta - 3\fint \frac{(\cos(4\theta)-\cos(4\eta))\cos(2\eta)}{2-2\cos(\theta-\eta)} d\eta \right)  \nonumber\\
  & = P_2 \left( -30\cos(2\theta)\cos(4\theta) - 3 \cos(6\theta)\right) \nonumber\\
  & = P_2 \left( -15\cos(2\theta) - 18\cos(6\theta) \right) \nonumber\\ 
  & = -15\cos(2\theta),
 \end{align}
 where the second equality follows from \eqref{lemma1} and Lemma~\ref{lemma3}. For $K_4$, we compute
 \begin{align}\label{k4}
 P_2 K_4 & = P_2 \left( -16\cos(4\theta) \fint \frac{\cos(2\theta)-\cos(2\eta)}{2-2\cos(\theta-\eta)} d\eta -16 \fint \frac{(\cos(2\theta)-\cos(2\eta))\cos(4\eta)}{2-2-\cos(\theta-\eta)} d\eta \right)  \nonumber \\
 & = P_2 \left( -16 \cos(2\theta)\cos(4\theta)  + 8 \cos(2\theta) - 8\cos(6\theta)  \right) \nonumber \\
 & = P_2 \left( -16\cos(6\theta) \right) \nonumber\\
 & = 0,
 \end{align}
 where the second equality follows from  \eqref{lemma1} and Lemma~\ref{lemma3}. For $K_5$, we compute
 \begin{align}\label{k5}
 P_2 K_5 & = -16P_2(\sin(2\theta) \sin(4\theta)) \nonumber \\
  & = -8P_2 \left(  \cos(2\theta) - \cos(6\theta)\right) \nonumber \\
  & = -8\cos(2\theta),
 \end{align}
 where the first equality follows from \eqref{lemma2}. Hence it follows from \eqref{k1}, \eqref{k2}, \eqref{k3}, \eqref{k4}  and \eqref{k5}  that
 \begin{align*}
 \frac{d^2}{dtds}Q\mathcal{F} = \left( 0 , P_2\frac{d^2}{dtds}\tilde{\mathcal{F}}_2 \right) = (0, -12\cos(2\theta)),
 \end{align*}
 which proves \eqref{value5}.

To prove \eqref{value6}, note that $\frac{1}{3}Q\frac{d^3}{dt^3} \mathcal{F} = \left( 0 , \frac{1}{3} P_2 \frac{d^3}{dt^3} \tilde{\mathcal{F}}_2 \right)$. We use Lemma~\ref{cubic_1} with $r=\cos(2\theta)$ and obtain
\begin{align*}
\frac{1}{3}P_2\frac{d^3}{dt^3} \tilde{\mathcal{F}}_2(2,tv) = & P_2 \left[ 4\fint \frac{(\cos(2\theta)-\cos(2\eta))(3\cos^2(2\theta) + 2\cos(2\theta)\cos(2\eta) + \cos^2(2\eta))}{2-2\cos(\theta-\eta)} d\eta  \right. \\
& \left. -8 \fint \frac{(\cos(2\theta)-\cos(2\eta))^3}{(2-2\cos(\theta-\eta))^2} d\eta \right. \\
& \left.  -16\sin(2\theta) \fint \frac{(\cos(2\theta) - \cos(2\eta))^2\sin(\theta-\eta)}{(2-2\cos(\theta-\eta))^2} d\eta \right. \\
& \left. + 8\fint \frac{ (3\cos^2(2\theta) -4\sin^2(2\theta))(\cos(2\theta)-\cos(2\eta))}{2-2\cos(\theta-\eta)} d\eta \right. \\ 
& \left.  +\left( 32 \cos(2\theta)\sin^2(2\theta) - 16\cos^3(2\theta)\right) \right] \\
& = P_2L_1 + P_2 L_2 + P_2 L_3 + P_2 L_4 +P_2 L_5.
\end{align*}
For $L_1$, we compute
\begin{align*}
P_2L_1 & = P_2 \left(  12\cos^2(2\theta)\fint \frac{\cos(2\theta)-\cos(2\eta)}{2-2\cos(\theta-\eta)} d\eta + 8 \cos(2\theta) \fint \frac{(\cos(2\theta) -\cos(2\eta))\cos(2\eta)}{2-2\cos(\theta-\eta)} d\eta \right. \\
& \left. \quad + 2\fint \frac{\cos(2\theta) - \cos(2\eta)}{2-2\cos(\theta-\eta)} d\eta + 2 \fint \frac{(\cos(2\theta) - \cos(2\eta))\cos(4\eta)}{2-2\cos(\theta-\eta)} d\eta \right) \\
& = P_2 \left( 12\cos^3(2\theta) + 8\cos(2\theta)\left(-\frac{1}{2} +\frac{1}{2}\cos(4\theta)\right) +2\cos(2\theta) + (-\cos(2\theta) + \cos(6\theta))\right)\\
& = P_2 (8\cos(2\theta) +6\cos(6\theta))\\
& = 8\cos(2\theta),
\end{align*}
where the second equality follows from \eqref{lemma1} and~\ref{lemma3}.
 For $L_2$, we use Lemma~\ref{lemma4} and obtain
 \begin{align*}
 P_2L_2 & = -8P_2\left( \frac{9}{4}\cos(2\theta) -\cos(6\theta) \right)\\
  & = -18\cos(2\theta).
 \end{align*}
 For $L_3$, we use Lemma~\ref{lemma5} and obtain
 \begin{align*}
 P_2L_3 = 16 P_2 \sin(2\theta)\sin(4\theta) = 8P_2(\cos(2\theta) - \cos(6\theta)) = 8\cos(2\theta).
 \end{align*}
 For $L_4$, we compute 
 \begin{align*}
 P_2L_4 & = P_2\left(24\cos^2(2\theta) \fint \frac{\cos(2\theta) - \cos(2\eta)}{2-2\cos(\theta-\eta)} d\eta - 32\sin^2(2\theta) \fint \frac{\cos(2\theta) - \cos(2\eta)}{2-2\cos(\theta-\eta)} d\eta \right) \\
 & = P_2(24\cos^3(2\theta) - 32\sin^2(2\theta)\cos(2\theta)) \\
 & = P_2( 10\cos(2\theta) + 14\cos(6\theta)) \\
 & = 10\cos(2\theta),
 \end{align*}
 where the second equality follows from ~\eqref{lemma1}.
 For $L_5$, it follows immediately that
 \begin{align*}
 P_2L_5 = P_2 \left( -4\cos(2\theta) - 12\cos(6\theta) \right) = -4\cos(2\theta).
 \end{align*}
 Collecting the above results, we obtain 
 \begin{align*}
 \frac{1}{3}P_2\frac{d^3}{dt^3} \tilde{\mathcal{F}}_2(2,tv) =4\cos(2\theta),
 \end{align*} 
 which implies \eqref{value6}.
 
 
  To prove \eqref{value8}, we use \eqref{value7} and \eqref{second_derivative_f2} in Lemma~\ref{second_derivative} with $g=0,\ \tilde{g} =2\cos(2\theta),\ r=\cos(2\theta)$ and $\tilde{r} = 0$ and obtain
  \begin{align*}
  P_2\frac{d^2}{dtds}\tilde{\mathcal{F}}_{2}(2,tv+s\tilde{v}) & = P_2 \left( -4\cos^2(2\theta) + 4\fint \frac{(\cos(2\theta)+\cos(2\eta))(\cos(2\theta)-\cos(2\eta))}{2-2\cos(\theta-\eta)} d\eta \right. \\
  & \left. \quad +8 \fint \frac{\cos(2\eta)\sin(2\theta)\sin(\theta-\eta)}{2-2\cos(\theta-\eta)} d\eta \right)  \\
  & = P_2 \left( -4\cos^2(2\theta) + 4\cos^2(2\theta) + 4\left( -\frac{1}{2} + \frac{1}{2}\cos(4\theta) \right) + 4\sin^2(2\theta) \right) \\
  & = 0,
  \end{align*}
  where the second equality follows from \eqref{lemma1}, \eqref{lemma2} and \eqref{lemma3}. This implies \eqref{value8}.
  
  To prove \eqref{value9}, it follows from \eqref{value4} and Lemma~\ref{derivative} that
  \begin{align*}
  \partial_bD\mathcal{F}(2,0)\left(-8\cos(4\theta),\frac{3}{2}\cos(4\theta)\right) = (3\sin(4\theta),2\cos(4\theta)).
  \end{align*}
  By projecting it to the image of $Q$, we obtain \eqref{value9}.
  
  To prove \eqref{value11}, we note that $Q\partial_b D\mathcal{F}(2,0)\tilde{v} = \left( 0, P_2\partial_b\frac{d}{dt}\tilde{\mathcal{F}}_2(2,t\tilde{v}) \right)\bigg|_{t=0}$. Hence it follows from \eqref{linearf2} in Lemma~\ref{derivative} and \eqref{value7} that $P_2\partial_b\frac{d}{dt}\tilde{\mathcal{F}}_2(2,t\tilde{v})\bigg|_{t=0} = \cos(2\theta)$. This implies \eqref{value11}.
\end{proof}

\subsection{Basic Integrals}

\begin{lemma}
For $\mathbb{N} \ni m\ge 1$, it holds that
\begin{align}
&\fint \frac{\cos(m\eta)\sin(\theta-\eta)}{2-2\cos(\theta-\eta)} d\eta = \frac{1}{2}\sin(m\theta), \label{lemma2}\\
&\fint \frac{\cos(m\theta) - \cos(m\eta)}{2-2\cos(\theta-\eta)} d\eta = \frac{m}{2}\cos(m\theta). \label{lemma1}
\end{align}
\end{lemma}
\begin{proof}
For \eqref{lemma2}, it is clear that $\fint \frac{\cos(m\eta)\sin(\theta-\eta)}{2-2\cos(\theta-\eta)} d\eta = \frac{1}{2}\fint\cos(m\eta)\cot{\left(\frac{\theta-\eta}{2}\right)}d\eta =\frac{1}{2}H(\cos(m\theta))(\theta)$, where $H$ denotes the Hilbert transform in the periodic domain. Therefore the result follows immediately since $H(\cos(m\theta))(\theta) = \sin(m\theta)$. 

 For \eqref{lemma1}, we recall that  $\fint \frac{f(\theta)-f(\eta)}{1-\cos(\theta-\eta)}d\eta = \Lambda f(\theta) =:( - \Delta)^{\frac{1}{2}}f(\theta)$. Thus \eqref{lemma1} follows immediately.
\end{proof}

\begin{lemma}\label{lemma3}
\begin{align}
&\fint \frac{(\cos(2\theta) - \cos(2\eta))\cos(4\eta)}{2-2\cos(\theta-\eta)} d\eta = -\frac{1}{2}\cos(2\theta) + \frac{1}{2}\cos(6\theta), \label{lemma31}\\
&\fint \frac{(\cos(2\theta) - \cos(2\eta))\cos(2\eta)}{2-2\cos(\theta-\eta)} d\eta = -\frac{1}{2} + \frac{1}{2}\cos(4\theta), \label{lemma32}\\
&\fint \frac{(\cos(4\theta) - \cos(4\eta))\cos(2\eta)}{2-2\cos(\theta-\eta)} d\eta = \cos(6\theta), \label{lemma33}\\
&\fint \frac{(\cos(2\theta)-\cos(2\eta))\sin(2\eta)}{2-2\cos(\theta-\eta)}d\eta= \frac{1}{2}\sin(4\theta).\label{lemma34}
\end{align}
\end{lemma}
\begin{proof}
We will show \eqref{lemma31} only. \eqref{lemma32}, \eqref{lemma33} and \eqref{lemma34} can be proved in the same way. For \eqref{lemma31}, one can write 
\begin{align*}
\fint \frac{(\cos(2\theta) - \cos(2\eta))\cos(4\eta)}{2-2\cos(\theta-\eta)} d\eta  &= \frac{1}{2}\Lambda (\cos(2\theta)\cos(4\theta))(\theta) - \frac{1}{2}\cos(2\theta)\Lambda(\cos(4\theta))(\theta)\\
& = -\frac{1}{2}\cos(2\theta) + \frac{1}{2}\cos(6\theta),
\end{align*}
where the last equality follows from \eqref{lemma1}.
\end{proof}

\begin{lemma}\label{lemma4}
\begin{align*}
\fint \frac{(\cos(2\theta) - \cos(2\eta))^3}{(2-2\cos(\theta-\eta))^2} d\eta = \frac{9}{4}\cos(2\theta) - \cos(6\theta).
\end{align*}
\end{lemma}
\begin{proof}
We compute
\begin{align*}
\fint & \frac{(\cos(2\theta) - \cos(2\eta))^3}{(2-2\cos(\theta-\eta))^2} d\eta  = \fint \frac{(-2\sin(\theta-\eta)\sin(\theta+\eta))^3}{\left(4\sin^2\left( \frac{\theta-\eta}{2}\right)\right)^2} d\eta = 4\fint -\frac{\cos^3\left(\frac{\theta-\eta}{2}\right)\sin^3(\theta+\eta)}{\sin\left(\frac{\theta-\eta}{2}\right)} d\eta\\
& = 4\fint \frac{\cos^3\frac{\eta}{2} \sin^3(2\theta+\eta)}{\sin\frac{\eta}{2}}d\eta  = 4\fint \frac{\cos^3\frac{\eta}{2}}{\sin\frac{\eta}{2}}\left( 3\sin^2(2\theta)\cos(2\theta)\cos^2\eta\sin\eta + \cos^3(2\theta)\sin^3\eta\right) d\eta\\
& = 12\sin^2(2\theta)\cos(2\theta)\fint \frac{\cos^3\frac{\eta}{2}\cos^2\eta\sin\eta}{\sin\frac{\eta}{2}} d\eta + 4\cos^3(2\theta)\fint \frac{\cos^3\frac{\eta}{2}\sin^3\eta}{\sin\frac{\eta}{2}} d\eta  = \frac{21}{4}\sin^2(2\theta)\cos(2\theta) + \frac{5}{4}\cos^3(2\theta)\\
& = \frac{9}{4}\cos(2\theta) - \cos(6\theta),
\end{align*}
which proves the lemma.
\end{proof}

\begin{lemma}\label{lemma5}
\begin{align*}
\fint \frac{(\cos(2\theta) - \cos(2\eta))^2\sin(\theta-\eta)}{(2-2\cos(\theta-\eta))^2} d\eta = -\sin(4\theta).
\end{align*}
\end{lemma}
\begin{proof}
Using the integration by parts, we compute 
\begin{align*}
\fint \frac{(\cos(2\theta) - \cos(2\eta))^2\sin(\theta-\eta)}{(2-2\cos(\theta-\eta))^2} d\eta = -2 \fint \frac{(\cos(2\theta)-\cos(2\eta))\sin(2\eta)}{2-2\cos(\theta-\eta)}d\eta.
\end{align*}
Therefore the result follows from \eqref{lemma34}.
\end{proof}

\section*{Acknowledgements}

JGS was partially supported by the European Research Council through ERC-StG-852741-CAPA. JP was partially supported by NSF through Grants NSF DMS-1715418, and NSF CAREER Grant DMS-1846745. JS was partially supported by NSF through Grant NSF DMS-1700180. YY was partially supported by NSF through Grants NSF DMS-1715418, NSF CAREER Grant DMS-1846745, and Sloan Research Fellowship.

\bibliographystyle{abbrv}
\bibliography{references}

\begin{tabular}{l}
\textbf{Javier G\'omez-Serrano}\\
{Department of Mathematics} \\
{Brown University} \\
{Kassar House, 151 Thayer St.} \\
{Providence, RI 02912, USA} \\ \\
{and} \\ \\ 
{Departament de Matem$\grave{a}$tiques i Inform$\grave{a}$tica} \\
{Universitat de Barcelona} \\
{Gran Via de les Corts Catalanes, 585} \\
{08007, Barcelona, Spain} \\
{Email: javier\_gomez\_serrano@brown.edu, jgomezserrano@ub.edu} \\ \\
\textbf{Jaemin Park} \\
{School of Mathematics, Georgia Tech} \\
{686 Cherry Street, Atlanta, GA 30332} \\
{Email: jpark776@gatech.edu} \\ \\
\textbf{Jia Shi}\\
{Department of Mathematics}\\
{Princeton University} \\
{409 Fine Hall, Washington Rd,}\\
{Princeton, NJ 08544, USA}\\
{Email: jiashi@math.princeton.edu}\\ \\
\textbf{Yao Yao} \\
 {School of Mathematics, Georgia Tech}\\
 {686 Cherry Street, Atlanta, GA 30332}\\
 {Email: yaoyao@math.gatech.edu}\\
\end{tabular}



\end{document}